\documentclass[letter,11pt, reqno]{amsart}
\usepackage[margin=3cm]{geometry}

\usepackage{multicol}

\usepackage{adjustbox}
\usepackage{graphicx}%
\usepackage{multirow}%
\usepackage{amsmath,amssymb,amsfonts}%
\usepackage{amsthm}%
\usepackage{mathrsfs}%

\usepackage{booktabs}%
\usepackage{natbib}
\usepackage{hyperref}

\usepackage{caption}
\usepackage{amssymb}
\usepackage{amsfonts}
\usepackage{bm}
\usepackage{enumerate}
\newcommand{\R}{\mathbb{R}}

\newcommand{\D}{\mathrm{D}}

\newcommand{\kM}{${\rm M}_p$}
\newcommand{\kC}{${\rm C}_p$}
\newcommand{\kS}{${\rm S}_p^k$}
\newcommand{\kD}{${\rm D}_p^\gamma$}

\newcommand{\dsum}{\displaystyle\sum}
\newcommand{\dmax}{\displaystyle\max}
\newcommand{\dmin}{\displaystyle\min}

\newcommand{\cone}{\rm cone}
\newcommand{\conv}{\rm conv}

\renewcommand{\;}{\quad}

\theoremstyle{thmstyleone}%
\newtheorem{theorem}{Theorem}

\newtheorem{proposition}[theorem]{Proposition}%
\newtheorem{lemma}[theorem]{Lemma}%

\theoremstyle{thmstyletwo}%
\newtheorem{remark}{Remark}%

\theoremstyle{thmstylethree}%
\newtheorem{definition}{Definition}%

\let\origmaketitle\maketitle
\def\maketitle{
	\begingroup
	\def\uppercasenonmath##1{} 
	\let\MakeUppercase\relax 
	\origmaketitle
	\endgroup
}

\begin{document}

\title[]{\huge On the Strength of Linear Relaxations in Ordered Optimization}

\author[V. Blanco, D. Laborda \MakeLowercase{and} M. Mart\'inez-Ant\'on]{
{\large V\'ictor Blanco$^{\dagger}$, Diego Laborda$^{\dagger}$, and Miguel Mart\'inez-Ant\'on$^{\dagger}$}\medskip\\
$^\dagger$Institute of Mathematics (IMAG), Universidad de Granada\\
\texttt{vblanco@ugr.es}, \texttt{dilagu@correo.ugr.es}, \texttt{mmanton@ugr.es}
}

\maketitle

\begin{abstract}
We study the conditions under which the convex relaxation of a mixed-integer linear programming formulation for ordered optimization problems, where sorting is part of the decision process, yields integral optimal solutions. Thereby solving the problem exactly in polynomial time. Our analysis identifies structural properties of the input data that influence the integrality of the relaxation. We show that incorporating ordered components introduces additional layers of combinatorial complexity that invalidate the exactness observed in classical (non-ordered) settings. In particular, for certain ordered problems such as the min--max case, the linear relaxation never recovers the integral solution. These results clarify the intrinsic hardness introduced by sorting and reveal that the strength of the relaxation depends critically on the ``proximity'' of the ordered problem to its classical counterpart: problems closer to the non-ordered case tend to admit tighter relaxations, while those further away exhibit substantially weaker behavior. Computational experiments on benchmark instances confirm the predictive value of the integrality conditions and demonstrate the practical implications of exact relaxations for ordered location problems.
\end{abstract}


\section{Introduction}\label{sec:intro}

Convex relaxations have emerged as powerful tools for addressing a wide range of challenging optimization problems, either exactly or approximately. In mathematical optimization, the \emph{relax-and-round} paradigm has become a canonical framework: starting from an optimization problem defined over a complex, non-convex feasible set, one first relaxes this set to a larger convex region in which the problem becomes tractable, and then rounds the resulting optimal solution to a feasible point in the original domain, possibly by adding cutting planes or incorporating branching strategies in the process. Such relaxations play a twofold role: (i) they can be solved efficiently, providing a principled and computationally accessible starting point for the rounding step, and (ii) the optimal value of the relaxation yields a provable bound on the true optimum, thereby enabling formal performance guarantees for the overall algorithm. In many cases, the non-convexity of the feasible set arises from integrality requirements on the decision variables, so that the relaxed convex set corresponds to their \emph{continuous relaxation}. A classical example is the use of linear programming (LP) relaxations for mixed-integer linear programming (MILP).

In parallel, discrete optimization has established itself as a versatile and rigorous framework for designing decision-making tools that address increasingly complex challenges in logistics, supply chain management, and machine learning. Unlike traditional heuristic-based methods, discrete optimization enables the principled formulation of models that naturally incorporate structural constraints, operational requirements, and multiple competing objectives. This flexibility has proven especially valuable in solving real-world problems such as vehicle routing, facility location, inventory control, and network design, where efficiency and robustness are critical. It has also demonstrated great potential in machine learning and artificial intelligence, where optimization-based formulations have led to models that are accurate, fair, and explainable. By offering fine-grained control over decision processes, optimization-based approaches allow, for example, the explicit balancing of cost, service quality, and sustainability, or the enforcement of logical and structural rules dictated by geography, capacity, and timing. As a result, discrete optimization is playing an increasingly central role in advancing modern systems, making them more adaptive, efficient, and resilient to uncertainty.

However, although discrete optimization and, in particular, MILP provide a natural and expressive framework for modeling problems in different fields, they are inherently challenging to solve due to their combinatorial nature and NP-hardness. Whereas LP relaxations of these problems can be solved in polynomial time, offering a tractable alternative with strong duality and well-understood structural properties. This computational advantage is particularly appealing in applications where scalability and responsiveness are critical. By determining when LP relaxations yield integral optimal solutions, we may overcome some of the limitations associated with MILP, unlocking the potential to solve large-scale optimization problems with both theoretical guarantees and practical efficiency.

Among the decision problems where this analysis becomes particularly relevant is \emph{location science}, a branch of operations research concerned with determining the optimal placement of facilities to serve spatially distributed demands under efficiency or equity criteria. In this context, ordered aggregation operators have been proposed to evaluate and improve the quality of selected placements, offering a robust alternative to the classical average-based criteria commonly used in facility location models. However, incorporating such operators into the resulting optimization formulations substantially increases the computational complexity of the problem, posing new challenges for both exact and heuristic solution methods.

Although the LP relaxation may serve as a computationally efficient surrogate for these problems, it remains an open question whether and when this relaxation can, in fact, lead to \emph{exact recovery}, that is, yield the optimal solution of the underlying discrete optimization problem. In this work, we provide a rigorous characterization of this phenomenon within the family of discrete ordered median problems. Specifically, we introduce an \emph{ordered contribution function} based on primal-dual certificates that allows us to determine when the LP relaxation recovers the integrality of the solution. Building on this framework, we analyze two fundamental properties of the input data, the presence of \emph{strongly sortable solutions} and the absence of \emph{equidistant points}, and show how they influence the exactness of the relaxation. Moreover, we identify and study particular subclasses of ordered median problems for which we derive \emph{sufficient conditions for non-recovery}, enabling the detection of those instances whose LP relaxations are inherently weak. Complementing the theoretical analysis, we conduct an extensive empirical study on benchmark instances from facility location, evaluating in practice the quality of the LP relaxations for several problem variants. To the best of our knowledge, this is the first computational analysis that examines the impact of the \emph{clusterability} of the input data on the behavior of LP relaxations in location problems. Altogether, our results provide new theoretical and empirical insights into the interplay between combinatorial structure and convex relaxation, clarifying when LP-based approaches can be reliably employed, or not,  for ordered optimization problems.

\section{Related Literature}

In this section, we provide a brief description of prior work related to the tools and problems analyzed in this paper. Our study combines two main ingredients: convex relaxations of challenging mathematical optimization problems and ordered operators.

Given the relevance of ordered optimization problems in location science, we focus in this paper on this field, although several of our results might be adapted to other application domains.


In the context of operations research, making optimal decisions typically requires formulating and solving challenging mathematical optimization problems, commonly in the shape of MILP. State-of-the-art methods, as branch-and-bound methods, rely on solving LP relaxations enhanced by the addition of cutting planes and branching strategies. The efficiency of solving a MILP often depends critically on the tightness of its LP relaxation, that is, on how closely the relaxed feasible region approximates the convex hull of feasible integral solutions. In some special cases, as when the constraint matrix is totally unimodular, solving the LP relaxation is sufficient, as its solution coincides with the optimal integral one.

Since MILPs are NP-hard in general~\citep[see, e.g.,][]{papadimitriou1981complexity}, while LP relaxations can be solved in polynomial time~\citep{leonid1979polynomial}, such relaxations provide tractable approximations that guide the solution process in branch-and-bound and cutting-plane algorithms. From a geometric perspective, the LP relaxation reveals the structure of the underlying polyhedron, enabling the identification of tight facets and valid inequalities that are instrumental in strengthening formulations. Moreover, analyzing the gap between the LP relaxation and the optimal integral solution yields valuable insights into the problem's complexity and helps prioritize which constraints or variables to refine. Overall, LP relaxations serve both as a practical tool for scalable computation and as a theoretical lens to understand the polyhedral geometry underlying discrete optimization problems.

The relax-and-round approach has been widely applied to construct high-quality solutions for several optimization problems. Its success depends both on the strength of the convex relaxation and on the effectiveness of the rounding procedure. This strategy has been particularly useful in facility location. For example, \citet{shmoys1997approximation} propose a rounding scheme for the uncapacitated facility location problem based on its LP relaxation. 

A common strategy to improve relaxations is to strengthen a \emph{basic} formulation by means of valid inequalities that tighten the feasible region toward the convex hull of integral solutions. For the $p$-median problem, polyhedral studies have identified families of valid inequalities, including cover, rank, and Chv\'atal--Gomory inequalities~\citep[see, e.g.,][among many others]{avella2007computational}. For the $p$-center problem, \citet{elloumi2004new} proposed a semi-relaxation framework, where further relaxation of certain variables can lead to tighter bounds. In the case of ordered median location problems, \citet{martinez2023constraint} proposed a constraint-relaxation approach where so-called \emph{strong order constraints} are added in a branch-and-cut fashion, demonstrating how carefully designed relaxations can drive the solution process.

An alternative but related line of research is based on \emph{Lagrangian relaxations}. Here, difficult constraints are relaxed and incorporated into the objective function with associated multipliers, which are iteratively updated until a termination criterion is reached. The resulting relaxed problems are easier to solve, and their solutions can be used to construct feasible solutions to the original problem. The tighter the relaxation, the better the resulting solutions. This approach has been successfully applied to different facility location problems~\citep[see, e.g.,][]{senne2000lagrangean,an2017lp}.

Beyond linear relaxations, one of the most successful types of convex relaxation-based algorithms includes semidefinite programming (SDP) relaxations. Over the last years, much work has been done to understand the phenomenon of integrality recovery in convex relaxation methods, both LP and SDP relaxations, above all in data science, such as clustering problems. Recent LP relaxations that achieve an integral solution to
clustering problems include \cite{awasthi2015relax,de2022ratio}, while some SDP relaxations that achieve
integral solutions are \cite{li2020birds,ames2014convex,awasthi2015relax}, among others.

In summary, the quality of a convex relaxation directly impacts the efficiency of algorithms for MILP, particularly in location problems, and the scalability of their application to large-scale instances. Of special interest is the identification of conditions on the input data under which the relaxation is \emph{exact}, that is, its optimal solution is already integral. Understanding these conditions not only yields structural insights into challenging optimization problems but also guides the design of specialized algorithms that can bypass expensive branching or cutting-plane procedures. This motivates our study of such exactness conditions, with the goal of bridging the gap between theoretical characterizations and algorithmic performance in practice. Moreover, we empirically investigate how the \emph{clusterability} of the input data influences the quality of LP relaxations in facility location problems. This analysis is inspired by recent probabilistic approaches in clustering, where the likelihood of integrality recovery has been studied under random geometric models with data points generated within Euclidean balls of known centers and radii~\citep[see, e.g.,][]{awasthi2015relax,r:del2023k}. Here, we aim to uncover, empirically, how geometric structure and spatial cohesion in the data affect the tightness of LP relaxations and their ability to recover integral solutions for different ordered median problems.


The other ingredient considered in this paper is ordered optimization. The ordered weighted averaging (OWA) operator, introduced by ~\cite{Yager1988}, is a flexible aggregation mechanism that generalizes many classical tendency measures in statistics. Given a finite set of real values, the OWA operator reorders them in non-increasing order and then applies a weighted average, where the weights are assigned not to the original positions of the values but to their ranks. This reordering step distinguishes OWA from standard weighted averages and enables it to model a wide range of preference attitudes. Depending on the choice of the weight vector, the OWA operator can reproduce the minimum, maximum, range, median, quantiles, or arithmetic mean, among others, making it a powerful and adaptive tool in multi-criteria decision analysis, robust optimization, and aggregation in machine learning and data science. 

One of the most popular uses of ordered operators is found in location science, the so-called \emph{ordered location problems}, as a unified methodology to cast different cost-based objective functions~\citep[see, e.g.,][]{g:puerto2000geometrical,nickel2005location,cont:blanco2014revisiting,blanco2016continuous,Labbe2017,Marin2020, blanco2023fairness,Ljubic2024}. Other fields where ordered  aggregations have been successfully applied include: voting problems~\citep{Ponce2018}, portfolio selection~\citep{Cesarone2024}, network design \citep{Puerto2015OrderedMH}, or linear regression~\citep{Blanco2021}. Nevertheless, the increasing interest in guiding decision-making tools through fair and explainable solutions has resulted in several works analyzing the convenience of this framework in different fields. For instance, the aggregation of residuals of linear regression models using ordered operators has given rise to the computation of robust estimators for these models in the presence of outliers~\citep[see, e.g.,][]{Yager2009, Blanco2021,Puerto2024}.

\subsection*{Contribution and Organization}

In this paper, we advance the study of integrality recovery from the linear programming relaxation within a family of challenging ordered location problems. Our starting point is the classical $p$-median problem, for which the structural conditions ensuring tightness of the LP relaxation have been extensively investigated in recent work~\citep[see][]{awasthi2015relax,r:del2023k}, leading to a comprehensive understanding of its combinatorial properties. We show, however, that such favorable behavior does not extend to ordered variants. In problems such as the $p$-center, the $p$-$k$-sum, or the $p$-$\gamma$-centdian, the additional sorting components introduce new layers of combinatorial complexity that break the exactness of the LP relaxation observed in the $p$-median case. Our results clarify the intrinsic difficulty of these ordered settings and highlight that the quality of the LP relaxation is strongly influenced by the ``proximity'' of a problem to the median operator: problems closer to the $p$-median tend to exhibit tighter relaxations, whereas those further away display substantially weaker behavior. 

We first analyze the theoretical foundations of these conclusions and subsequently conduct an extensive computational study that not only provides empirical evidence supporting our results in practice but also highlights new open questions for further research on the topic. In particular, following the line of previous research on the $p$-median problem, we examine the impact of structured instances, such as their \emph{clusterability} (i.e., the extent to which the input admits a clustered structure), on the strength of the relaxations. Although theoretical properties have been derived for clustering algorithms based on the (extended) stochastic ball model~\citep{r:del2023k}, our theoretical results indicate that the recovery property is very restrictive for ordered location problems, and for many choices of sorting weights, it will never hold. Hence, our empirical study sheds light on the impact that favorable geometric distributions of points may have on the quality of the LP relaxation of the problem.

The remainder of the paper is organized as follows. Section~\ref{sec:domp} introduces in detail the family of ordered location problems studied here. We also recall the two main mixed-integer linear programming formulations from the literature for solving these problems, namely those proposed in \citep{cont:blanco2014revisiting,Ogryczak2003}, which will serve as our oracles for the subsequent analysis on the tightness of their relaxations. In Section~\ref{sec:primal-dual}, we present the LP relaxation for ordered problems as well as its dual. We define the ordered contribution function (Definition~\ref{def:ocf}), which can be interpreted as the overall contribution that a point receives from all others. Using this function, we characterize the exactness of the LP relaxation by means of the optimality conditions of the primal--dual pair (Theorem~\ref{th:ocf}), and we subsequently exploit this characterization for structured instances. Section~\ref{sec:specialcases} is devoted to special cases of interest that have been widely studied in the literature, culminating in Theorem~\ref{th:specialcases}, which highlights how sorting affects the exactness of the linear relaxation within this family of location problems. Section~\ref{sec:exps} provides empirical evidence supporting our findings through a wide range of computational experiments on benchmark instances, illustrating the actual effect of sorting on the tightness of the linear relaxations. In particular, we empirically address the following questions: is there a monotone relationship between the spatial configuration of the input data and computational effort? Do more structured instances lead to faster solutions or tighter relaxations? Which ordered problems are positively affected by the input structure? Does sorting matter for the quality of the LP relaxations?

Finally, Section~\ref{sec:conclusion} presents our conclusions and outlines directions for future research on the topic.

\section{The Discrete Ordered Median Problem}\label{sec:domp}

Let $P= \{a_1, \ldots, a_m\} \subset \R^n$ a finite set of points in the $n$-dimensional real space. We denote by $[m] := \{1,\ldots, m\}$ the index sets for the points in $P$. We are also given a metric $\D: P\times P\to \R_+$, completely determined by the matrix $(d_{ab})\in \R^{P\times P}_+$, where $d_{ab}:=\D(a,b)$. The use of natural indices may be used as appropriate, being $d_{ij}:=d_{a_i a_j}$.

The goal of the discrete $p$-ordered median problem (DOMP, for  short) for the set of points in $P$ is to select a set of centers $Q \subseteq P$ with $|Q|=p$  optimizing a measure on the quality of those centers, usually an aggregation of the distances from the points to their closest center in $Q$. Thus, for a given subset $Q \subseteq P$, and each $j \in [m]$, we denote by $d_j := \dmin_{c\in Q}d_{a_j c}$, i.e., the closest distance from $a_j$ to the points in $Q$. 

With this notation, in what follows, we define the key measure that will be used in our paper.

\begin{definition}[Ordered Median Operator]
    Let $\bm{\lambda} = (\lambda_1, \ldots, \lambda_m)\in \R^m$ and $Q\subseteq P$. The \emph{ordered median} (OM) \emph{operator} on the cost vector $(d_j)_{j \in [m]}$ is defined as:
\begin{equation}\label{eq:om}\tag{\rm OM}
    {\rm OM}_{\bm{\lambda}}(Q):=\dsum_{j=1}^m\lambda_j d_{\sigma(j)}.
\end{equation}
where $\sigma \in \mathscr{S}_m$ is a permutation of the indices of the points in $P$, such that $d_{\sigma(j)}\geq d_{\sigma(j+1)}$ for $j\in [m-1]$. 
\end{definition}

The goal of a $p$-DOMP problem is to find a subset $Q \subseteq P$ with $|Q| = p$  (the $p$ centers) that minimizes an ordered median \eqref{eq:om} aggregation of the distances $(d_j)_{j \in [m]}$, using a given weight vector $\bm{\lambda}$.

\begin{equation}\label{eq:domp}\tag{$p$-{\rm DOMP}}
    \min_{\substack{Q \subseteq P \\ |Q| = p}} \ref{eq:om}_{\bm{\lambda}}(Q).
\end{equation}

While particular instances of this problem have been studied in the literature, we address it in its general form, presenting a unified framework that enables decision-making under different values of the hyperparameter $\bm{\lambda}$, which should be selected based on the criteria, the priorities, and the characteristics that shape the decision, the decision maker, and the input dataset.

The use of OM operators in location science introduces a powerful and flexible mechanism to balance between different objectives. By specifying a non-increasing weight vector $\bm\lambda$, one can smoothly interpolate between the minimization of the maximum distance (as in $p$‑center), the sum of distances (as in $p$‑median), or any intermediate compromise. This flexibility allows practitioners to control the degree of robustness against anomalous points: placing more emphasis on larger distances yields a solution closer to $p$‑center, while a flatter weight distribution mimics the behavior of $p$‑median \cite[see, e.g.,][for illustrative examples of the different solutions that can be obtained with different choices of $\bm\lambda$]{disc:blanco2019ordered}. Moreover, when the variables are relaxed to convex domains, choosing a non-increasing $\bm\lambda$ ensures that the resulting optimization problem remains convex and tractable under metric-based cost vectors. Consequently, OM‑based location offers a unified, interpretable, and computationally efficient framework to adapt locational behaviors to the specific characteristics and needs of the users.

The most \emph{natural} mathematical optimization approach for the problem is to consider sets of binary variables: $(y)$ for the selection of centers and $(x)$ for the assignment and sorting the distances,

\adjustbox{width=\textwidth}{\begin{tabular}{cc}
$
   y_j = \begin{cases}
    1 & \mbox{if $a_j$ is selected as center},\\
    0 & \mbox{otherwise}
\end{cases} 
$
&  $
    x_{ijr} = \begin{cases}
        1 & \mbox{if $a_i$ is assigned to center $a_j$ and $d_i$}\\
        & \mbox{is sorted in $r$th position},\\
        0 & \mbox{otherwise}
    \end{cases}
$
\end{tabular}}
for $i, j, r \in [m]$.

Being the following one a valid formulation for the problem~\citep[see, e.g.,][for further details]{disc:boland2006exact}:
\begin{align*}
    \min & \; \sum_{r=1}^m \lambda_r \sum_{i=1}^m\sum_{j=1}^m d_{ij} x_{ijr} \\
    \mbox{s.t.} & \; \sum_{j=1}^m\sum_{r=1}^m x_{ijr} =1, & \forall i \in [m];\\
    & \; \sum_{i=1}^m\sum_{j=1}^m x_{ijr} =1, & \forall r \in [m];\\
    & \; \sum_{r=1}^m x_{ijr} \leq y_j, & \forall i,j \in [m];\\
    & \; \sum_{j=1}^m y_j  = p, \\
    & \; \sum_{i=1}^m\sum_{j=1}^m d_{ij} x_{ijr}  \geq \sum_{i=1}^m\sum_{j=1}^m d_{ij} x_{ij(r+1)}, & \forall r \in [m-1];\\
    & \; y_j \in \{0,1\}, & \forall  j\in [m];\\
    & \; x_{ijr} \in \{0,1\}, & \forall i,j, r \in [m].
\end{align*}
\cite{disc:boland2006exact} proposed alternative formulations and strengthening techniques for this problem that have been applied to other types of problems where sorting distances is part of the decision~\citep{blanco2016continuous,disc:blanco2019ordered}.

We assume in this paper that the ${\bm \lambda}$-weights are nonnegative, whereas the convexity of the ordered median operator is assured without such an assumption. Nevertheless, when this operator is inserted into the minimization model, unless one explicitly requires that each point is allocated to its closest center, the optimal solution would allocate the points that are sorted in the position/s of the negative $\lambda$'s to their farthest center to obtain smaller (negative) values of the objective function. This might be solved by incorporating closest-assignment constraints, as those proposed by \cite{espejo2012closest} (see Remark \ref{closest}).

Furthermore, although in the $p$-median the optimization problem that results when fixing the centers in $Q$ is known to be convex, it is not always the case for \ref{eq:domp}, as we prove in the following result.
\begin{lemma}
    \ref{eq:domp} is convex if and only if $\lambda_1 \geq \cdots \geq \lambda_m$.
\end{lemma}
\begin{proof} 
Denoting by $\theta_r(\bm{d}) = \sum_{\ell=1}^r d_{\sigma(\ell)}$, for $r\in [m]$, and $\theta_m({\bm d})=\sum_{\ell=1}^m d_{\ell}$; the ordered median operator \eqref{eq:om} can be equivalently written as:
    \begin{equation*}
        {\rm OM}_{\bm \lambda}({\bm d}) = \sum_{r=1}^m \Delta_r \theta_r({\bm d})
    \end{equation*}
    where $\Delta_r = \lambda_r - \lambda_{r+1}$ for $r\in [m-1]$, and $\Delta_m = \lambda_m$. Note that $\theta$-functions are sublinear ($\theta_m$ is a linear function). Now, since $\lambda_1 \geq \cdots \geq \lambda_m$, then $\Delta_r \geq 0$ for all $r\in [m-1]$ and ${\rm OM}_{\bm \lambda}$ is a nonnegative linear combination of sublinear functions plus a linear function, so it is sublinear and then convex.

    Conversely, suppose $\lambda_s < \lambda_{s+1}$ for some $s \in [m]$. Then, we could consider the cost vectors
    \begin{equation*}
        {\bm c} = (\overbrace{1,\dots,1}^{s-1},1,0,\overbrace{0,\dots,0}^{ m-s-1}) \quad \text{and} \quad {\bm d} = (\overbrace{1,\dots,1}^{s-1},0,1,\overbrace{0,\dots,0}^{ m-s-1});
        \end{equation*}
    then
    \begin{equation*}
        {\rm OM}_{\bm \lambda}({\bm c}+{\bm d}) = {\rm OM}_{\bm \lambda}({\bm c}) + {\rm OM}_{\bm \lambda}({\bm d}) + \lambda_{s+1} - \lambda_s > {\rm OM}_{\bm \lambda}({\bm c}) + {\rm OM}_{\bm \lambda}({\bm d}).
    \end{equation*}
Thus, ${\rm OM}_{\bm \lambda}$ is not subadditive, and then, since it is positively homogeneous, it cannot be convex~\cite[Theorem 4.7]{rockafellar1970convex}.
\end{proof}

Thus, from now on, we assume that the weights $\bm{\lambda}$ are sorted in non-increasing order, to assure convexity of the relaxed problem. In this case, alternative formulations have been derived to handle this situation more efficiently, avoiding the need to incorporate binary variables for representing the permutations for sorting the distances in the mathematical optimization model~\citep[see, e.g.,][]{cont:blanco2014revisiting,Ogryczak2003}. In particular, in \citep{cont:blanco2014revisiting}, the key observation is that in case the weights are sorted in non-increasing order, the OM operator for a fixed set of centers $Q$ is equivalent to
\begin{equation}\label{eq:convexom}
    \ref{eq:om}_{\bm{\lambda}}(Q)=\dmax_{\sigma\in \mathscr{S}_m} \dsum_{r=1}^m\lambda_r d_{\sigma(r)}
\end{equation}
where $\mathscr{S}_m$ is the group of permutations of the index set $[m]$.

This representation allows for the following formulation of the problem that, apart from the $y$-variables defined above, uses the classical two-index assignment variables:
$$
z_{ij} = \begin{cases}
    1 & \mbox{if $a_i$ is assigned to center $a_j$,}\\
    0  & \mbox{otherwise}
\end{cases}
$$
as well as two sets of continuous variables $u_i, v_r \in \R$ for $i, r \in [m].$
\begin{lemma}[\cite{cont:blanco2014revisiting}]\label{lemma:bep}
Let $\bm{\lambda}$ such that $\lambda_1 \geq \cdots \geq \lambda_m \geq 0$. Then, \ref{eq:domp} can be solved with the following mixed-integer linear programming problem:
\begin{align}\label{eq:milp-bep}\tag*{\ensuremath{p\text{-}{\rm DOMP}^{\bm{\lambda}}_{\rm BEP}(\D)}}
		\min & \; \sum_{i=1}^m u_i + \sum_{r=1}^m v_r & \\
			\mbox{s.t.} & \; \sum_{j=1}^m z_{ij} = 1,& \forall i\in [m];  \label{ctr:1}\\
			    & \; z_{ij} \leq y_j, & \forall i, j \in [m]; \label{ctr:2}\\
			    & \; \sum_{j=1}^m y_j = p, \label{ctr:3} \\
			    & \; u_i + v_r \geq \lambda_r \sum_{j=1}^m z_{ij} d_{ij}, & \forall i, r \in [m]; \label{ctr:4} \\
		          & \; y_j, z_{ij} \in \{ 0,1 \}, & \forall i, j \in [m]; \label{domp:integrality}\\
                & \; u_i ,v_r \in \R, & \forall i, r\in [m]. \nonumber
	\end{align}
\end{lemma}

We denote by:
\begin{equation*}
    \mathcal{P}_{\rm IP} := \{(y, z, u, v) \in \{0,1\}^{m} \times \{0,1\}^{m\times m} \times \R^{m} \times \R^m: (y, z, u, v) \text{ verifies } \eqref{ctr:1}\text{-}\eqref{ctr:4} \}
\end{equation*} 
the feasible region of the problem above.

\begin{remark}[Closest Assignment]\label{closest}
  As already mentioned, we assumed the nonnegativity of the ${\bm{\lambda}}$-weights, since in that case there is an optimal solution for the problem where the points are allocated to their closest open centers. If we skip this assumption, one can ensure this condition by imposing any of the available closest-assignment constraints in location science. Among those of size $O(m^2)$, the one that dominates others was proposed by  \cite{espejo2012closest}, where for all $i, j \in [m]$ the closest-assignment constraint reads for \ref{eq:milp-bep}:
\begin{align*}
\sum_{\substack{j'\in [m]:\\ d_{ij'} < d_{ij} \\ |\Theta_{ij}| - |\Theta_{ij'}| \leq p}}\!\!\!\!\!\!Q_{ijj'} z_{ij'}
+ P_{ij} \!\!\sum_{\substack{j'\in[m]:\\ d_{ij'} \geq d_{ij}}}\!\!\!\! z_{ij'}
+ \sum_{\substack{j'\in [m]:\\ d_{ij'} < d_{ij}}}\!\! y_{j'}
\leq P_{ij},
\end{align*}
where $\Theta_{ij} = \{\,l\in [m]: d_{il} < d_{ij}\,\}$, $P_{ij} = \min \{p, |\Theta_{ij}|\}$, and $Q_{ijj'} = |\Theta_{ij'}| + \min\{0, p - |\Theta_{ij}|\}$, for all $i, j, j' \in [m]$.
\end{remark}

On the other hand, although the computational complexity of solving the MILP formulation above remains challenging, it allows for a reduction in the number of binary variables compared to the general formulation. The main goal of this paper is to analyze how tight the linear programming relaxation of the problem is under certain conditions on the point-to-center assignments and the $\bm{\lambda}$-weights. Specifically, we are interested in understanding when relaxing the integrality constraints \eqref{domp:integrality} leads to an integral solution of the problem, and hence to an optimal solution of \ref{eq:domp}, which could then be obtained in polynomial time by solving a linear program instead of the more complex MILP. Consequently, we use the LP relaxation of \ref{eq:milp-bep} as a polynomial-time oracle for solving \ref{eq:domp}. 

An alternative MILP model for the problem was presented by \cite{Ogryczak2003} based on the $k$-sum approach:
\begin{align}
    \min & \; \sum_{s=1}^m \Delta_s \left(s t_s + \sum_{i=1}^m q_{is}\right) \label{eq:milp-ot}\tag*{\ensuremath{p\text{-}{\rm DOMP}^{\bm{\lambda}}_{\rm OT}(\D)}}\\
    \mbox{s.t.} & \; \sum_{j=1}^m z_{ij} = 1, & \forall i\in [m]; \nonumber \\
			    & \; z_{ij} \leq y_j, & \forall i, j \in [m]; \nonumber\\
			    & \; \sum_{j=1}^m y_j = p, &\nonumber \\
                & \; q_{is} \geq \sum_{i=1}^m d_{ij} z_{ij} - t_s, &\forall i, s \in [m];\nonumber\\
                & \; t_s \in \R, &\forall s \in [m];\nonumber\\
                & \; q_{is} \in \R, &\forall i, s \in [m];\nonumber\\
                & \; y_j, z_{ij} \in \{ 0,1 \}, & \forall i, j \in [m]; \nonumber
\end{align}
where $\Delta_s = (\lambda_{s}-\lambda_{s+1})$ if $s <m$ and $\Delta_m = \lambda_m$.

The following result, proved by \cite{Marin2020}, states that the LP relaxations of both models are equivalent.
\begin{lemma}
    The objective values of the linear programming relaxations of \ref{eq:milp-bep} and \ref{eq:milp-ot} coincide.
\end{lemma}

The above result implies that the polyhedra of solutions induced by the LP relaxations of both problems are equivalent, in the sense that from a relaxed solution of one problem, one can construct a feasible solution of the other. Thus, the \emph{relaxation properties} of both oracles for solving the problem are equivalent.

Note that having a tight convex relaxation of a MILP model significantly improves the efficiency of solving the problem numerically. A tight relaxation provides a stronger bound on the optimal value, which can reduce the size of the branch-and-bound tree and accelerate convergence to an optimal integral solution. It also helps in pruning suboptimal regions of the search space more effectively and may even lead to integral solutions directly from the relaxation, eliminating the need for combinatorial search in some cases. Thus, analyzing the conditions when the MILP can be solved using LP also allows us to take a step forward in the design of algorithms for efficiently solving these challenging models.

\section{Integrality Recovery in the DOMP}\label{sec:primal-dual}

In this section, we analyze the theoretical conditions that ensure integrality recovery in~\ref{eq:domp} based on the linear relaxation of~\ref{eq:milp-bep}. First, we study the general conditions for recovery using primal--dual certificates. Then, under additional assumptions, we establish more concise results that, in turn, allow us to derive negative results regarding recovery.

\subsection{Primal--Dual Certificates: The Ordered Contribution}

To derive the theoretical certificates for integrality recovery in the \ref{eq:domp}, we denote by 
$$
    \mathcal{P}_{\mathrm{LP}} := \{(y, z, u, v) \in [0,1]^{m} \times  [0,1]^{m\times m} \times \R^{m} \times \R^m: (y, z, u, v) \text{ verifies } \eqref{ctr:1}\text{-}\eqref{ctr:4}\}
$$
the convex-relaxation of the feasible region of ~\ref{eq:milp-bep}, i.e. $\mathcal{P}_{\mathrm{IP}} = \mathcal{P}_{\mathrm{LP}} \cap \{0,1\}^{m} \times  \{0,1\}^{m\times m} \times \R^{m} \times \R^m$.

We consider then the LP relaxation of \ref{eq:milp-bep} that we use as our relaxation oracle:
\begin{align}
    \min & \; \sum_{i=1}^m u_i + \sum_{r=1}^m v_r \label{eq:lp-bep}\tag*{\ensuremath{p\text{-}{\rm DOMP}^{\bm{\lambda}}_{\rm LP}(\D)}}\\
    \mbox{s.t.} & \; (y,z,u,v) \in \mathcal{P}_{\mathrm{LP}} \nonumber.
\end{align}

    \begin{definition}[Integrality Recovery in DOMP]
\ref{eq:domp} with input data $(P,\bm{\lambda},p)$ is said to be LP \emph{recovered} if \ref{eq:lp-bep} 
admits an optimal solution $(y^\star,z^\star,u^\star,v^\star) \in \mathcal{P}_{\rm IP}$.
\end{definition}

Our guarantees for the LP to recover \ref{eq:domp} will be based on the primal-dual optimality conditions. It is easy to derive the dual of \ref{eq:lp-bep}:
\begin{align}\label{eq:lp-d}\tag*{\ensuremath{p\text{-}{\rm DOMP}^{\bm{\lambda}}_{\rm LD}(\D)}}
    \max & \; \sum_{i=1}^m \alpha_i - p \omega \nonumber\\
    \mbox{s.t.} & \; \alpha_i \leq \beta_{ij} + \sum_{r=1}^m \lambda_r \sigma_{ir} d_{ij}, & \forall i, j \in [m]; \nonumber\\
        & \; \sum_{i=1}^m \beta_{ij} \leq \omega, & \forall j\in [m]; \nonumber\\
        & \; \sum_{i=1}^m \sigma_{ir} = 1, & \forall r\in [m]; \label{eq:tum1}\\
        & \; \sum_{r=1}^m \sigma_{ir} = 1, & \forall i\in [m]; \label{eq:tu2}\\
        & \; \beta_{ij}, \sigma_{ir} \geq 0, & \forall i, j, r\in [m]; \nonumber\\
        & \; \alpha_i, \omega \in \R, & \forall i\in [m].\nonumber
\end{align}

Following the line of~\cite{awasthi2015relax} and~\cite{r:del2023k}, and as a generalization of the so-called \emph{contribution function}, we introduce the \emph{ordered contribution function}, defined as follows:
\begin{definition}[Ordered Contribution Function]\label{def:ocf}
    Let $\alpha \in \R^m$ and $\sigma \in \R^{m\times m}$ be a permutation matrix. The \emph{ordered contribution function}, $C_\sigma^\alpha : [m] \to \R_+$, is defined by
    \begin{equation}\label{eq:ocf}
        C_\sigma^\alpha(j) := \dsum_{i=1}^m \left( \alpha_i - \sum_{r=1}^m \lambda_r \sigma_{ir} d_{ij} \right)_+
    \end{equation}
    where $t_+$ stands for the positive part of a real $t$, i.e., $t_+ := \max\left\{t,0\right\}$. 
\end{definition}

In the above definition, a demand point $a_j \in P$ receives a positive \emph{ordered contribution} from those points $a_i \in P$ that are at a weighted distance ($\lambda_r d_{ij}$) smaller than $\alpha_i$, when $i$ is sorted by $\sigma$ in the $r$th position. In this sense, a point $a_i$ can only \emph{see} other points within a weighted distance $\alpha_i$. The values of 
$$
\beta_{ij} = \left( \alpha_i - \sum_{r=1}^m \lambda_r \sigma_{ir} d_{ij} \right)_+
$$
can thus be interpreted as the ordered contribution from $a_i$ to $a_j$. According to this observation, the ordered contribution function~\eqref{eq:ocf} represents the overall contribution that a point $a_j \in P$ receives from all points in $P$ in the ordered setting.

In what follows, we provide conditions for exact LP recovery in \ref{eq:milp-bep}. The characterization we present below recodes the optimality conditions of the primal-dual problem via the ordered contribution function~\eqref{eq:ocf}, allowing for an interpretation of how a given distribution of the input data affects the tightness of the LP relaxation.

\begin{theorem} \label{th:ocf}
    Let $(\bar{y}, \bar{z}, u, v)$ be a feasible solution to \ref{eq:milp-bep}, and $Q:=\{j\in [m] : \bar{y}_j=1\}$. For each $j\in Q$, let $\Gamma_j:=\{i\in [m]: \bar{z}_{ij}=1\}$. If  $(\bar{y}, \bar{z}, u, v)$ is optimal solution to \ref{eq:lp-bep} then there exists $\alpha \in \R^m$ and $\sigma \in \R^{m \times m}$ such that
    \begin{align}
		& C_\sigma^\alpha(j) = C_\sigma^\alpha(j'), & \forall j, j'\in Q; \label{eq:th1} \\
		& C_\sigma^\alpha(i) \leq C_\sigma^\alpha(j), & \forall i \in [m] \setminus Q, j\in Q; \label{eq:th2} \\
		& \alpha_i \geq \sum_{r=1}^m \lambda_r \sigma_{ir} d_{ij}, & \forall j\in Q, i \in \Gamma_j; \label{eq:th3} \\
		& \alpha_i \leq \sum_{r=1}^m \lambda_r \sigma_{ir} d_{ij}, & \forall j\in Q, i \in [m] \setminus \Gamma_j. \label{eq:th4}
	\end{align}

    Conversely, if $\sigma$ is a permutation verifying $d_{\sigma(r)}\geq d_{\sigma(r+1)}$ for $r\in [m-1]$, and there exists $\alpha\in \R^m$ satisfying \eqref{eq:th1}-\eqref{eq:th4}. Then, there exist $\tilde{u}, \tilde{v}$ such that $(\bar{y}, \bar{z}, \tilde{u}, \tilde{v})$ is optimal solution to \ref{eq:milp-bep}.

    Besides, whether inequalities \eqref{eq:th2}-\eqref{eq:th4} are satisfied strictly, every optimal solution $(y',z',u',v')$ to \ref{eq:lp-bep} fulfills $y'=\bar{y}$ and $z'=\bar{z}$.
\end{theorem}
\begin{proof}
First, observe that if $(y, z, u, v)$ and $(\alpha, \beta, \omega, \sigma)$ are feasible solutions of \ref{eq:lp-bep} and \ref{eq:lp-d}, respectively, by  complementary slackness theorem, both solutions are optimal to their respective problems if and only if they verify:
\begin{align}
	& \beta_{ij} (z_{ij} - y_j) = 0, & \forall i, j \in [m];\label{eq:cs1} \\
	& z_{ij} \left(\alpha_i - \beta_{ij} - \sum_{r=1}^m \lambda_r \sigma_{ir} d_{ij}\right) = 0, & \forall i, j \in [m]; \label{eq:cs2} \\
	& y_j \left( \sum_{i=1}^m \beta_{ij} - \omega \right) = 0, & \forall j \in [m]; \label{eq:cs3} \\
	& \sigma_{ir} \left( u_i + v_r - \lambda_r \sum_{j=1}^m z_{ij} d_{ij} \right) = 0, & \forall i, r \in [m]. \label{eq:cs4}
\end{align}

In case $(\bar{y}, \bar{z}, u, v)$ is solution to \ref{eq:milp-bep}, Conditions~\eqref{eq:cs1}-\eqref{eq:cs3}\footnote{Note that, if $\sigma$ takes binary values, when sorting the distance vector induced by $\bar{y}$ and $\bar{z}$, one can always take $u$ and $v$ such that $u_i+v_r=\lambda_rd_{ij}$ where $a_i$ is assigned to $a_j$ and $d_i$ ($=d_{ij}$) is sorted in $r$th position (Lemma~\ref{lemma:bep}). Therefore, \eqref{eq:cs4} holds too. This fact motivates the notion of \emph{strongly sortable solutions} introduced below.} simplify to:
\begin{align}
	& \beta_{ij} = 0, & \forall i, j \in [m] \text{ such that } y_j = 1, z_{ij} = 0; \label{eq:condition1} \\
	& \beta_{ij} = \alpha_i - \sum_{r=1}^m \lambda_r \sigma_{ir} d_{ij}, & \forall i, j \in [m] \text{ such that } z_{ij} = 1; \label{eq:condition2} \\
	& \sum_{i=1}^m \beta_{ij} = \omega, & \forall j \in [m] \text{ such that } y_j = 1.\label{eq:condition3}
\end{align}

Note that if $(\alpha,\beta,\omega,\sigma)$ is a feasible solution of \ref{eq:lp-d}, since $\beta_{ij}$ does not appear in the objective function, one can replace its value by  $\beta'_{ij} = \left( \alpha_i - \sum_{r=1}^m \lambda_r \sigma_{ir} d_{ij} \right)_+$, for all $i, j \in [m]$.

With the above observations, we are now ready to prove the result. Let us then assume that $(\bar{y}, \bar{z}, u, v)$ is an optimal solution to \ref{eq:lp-bep}, and  $(\alpha, \beta, \omega, \sigma)$ its dual solution. We denote by $\beta'_{ij}:=\left( \alpha_i - \sum_{r=1}^m \lambda_r \sigma_{ir} d_{ij} \right)_+$ for every $i, j\in [m]$. By definition, $\sum_{i=1}^m\beta'_{ij}=C^{\alpha}_\sigma(j)$ for all $j\in [m]$. Hence, \eqref{eq:condition3} implies $C_\sigma^\alpha(j) = \sum_{i=1}^m \beta'_{ij} = \omega$ for all $j\in Q$, so the first condition \eqref{eq:th1} is fulfilled. The second constraint of \ref{eq:lp-d} ensures $C_\sigma^\alpha(i) = \sum_{t=1}^m \beta'_{ti} \leq \omega = C_\sigma^\alpha(j)$, for all $i\in [m]$ and $j\in Q$, therefore \eqref{eq:th2} is verified. Equation \eqref{eq:condition2} implies that $\alpha_i - \sum_{r=1}^m \lambda_r \sigma_{ir} d_{ij} = \beta'_{ij}$ if $i \in \Gamma_j$. Since $\beta'_{ij} \geq 0$, we also have that $\alpha_i \geq \sum_{r=1}^m \lambda_r \sigma_{ir} d_{ij}$, thus \eqref{eq:th3} is satisfied. The first constraint of  \ref{eq:lp-d} assures $\alpha_i \leq \beta'_{ij} + \sum_{r=1}^m \lambda_r \sigma_{ir} d_{ij}$ for every $i \in [m]$ and $j\in Q$. Thus, if $i \notin \Gamma_j$, Equation \eqref{eq:condition1} implies $\beta'_{ij} = 0$, hence $\alpha_i \leq \sum_{r=1}^m \lambda_r \sigma_{ir} d_{ij}$ and the last condition \eqref{eq:th4} becomes true.

Now we will see the converse. Let $\sigma$ be a permutation matrix in the hypothesis of the theorem, $\beta_{ij}:= \left( \alpha_i - \sum_{r=1}^m  \lambda_r \sigma_{ir} d_{ij} \right)_+$, and $\omega := C_\sigma^\alpha(j)$ for $j\in Q$. By construction, $(\alpha,\beta,\omega,\sigma)$ is a feasible solution to~\ref{eq:lp-d}. By hypothesis, $\sigma$ is an optimal solution to \eqref{eq:convexom}, hence by Lemma~\ref{lemma:bep} there is $(\tilde{u}, \tilde{v})$ able to satisfy \eqref{ctr:4} and \eqref{eq:cs4} for $\sigma$ and $\bar{z}$. The last condition \eqref{eq:th4} says that if $\bar{y}_j=1$ and $\bar{z}_{ij}=0$, then we have $\alpha_i - \sum_{r=1}^m \lambda_r \sigma_{ir} d_{ij} \leq 0$ and by definition $\beta_{ij}=0$, therefore \eqref{eq:condition1} is verified. Condition \eqref{eq:th3} says that, if $\bar{z}_{ij} = 1$, then $\alpha_i - \sum_{r=1}^m \lambda_r \sigma_{ir} d_{ij} \geq 0$, thus, by definition $\beta_{ij} = \alpha_i - \sum_{r=1}^m \lambda_r \sigma_{ir} d_{ij}$ and \eqref{eq:condition2} is given. Straightforwardly, Equation \eqref{eq:condition3} is fulfilled since $C_\sigma^\alpha(j) = \sum_{i=1}^m \beta_{ij} = \omega$ fo every $j\in Q$.

We suppose now conditions \eqref{eq:th2}-\eqref{eq:th4} are satisfied strictly. Let $(y',z',u',v')$ an optimal solution to \ref{eq:lp-bep}. We take $(\alpha,\beta,\omega,\sigma)$ an optimal solution to \ref{eq:lp-d}. Let $i\in [m]\setminus Q$, we have that $C^\alpha_\sigma(i) = \sum_{t=1}^m \beta_{ti} < \omega$ because of \eqref{eq:th2} and hypothesis, therefore $\sum_{t=1}^m \beta_{ti} - \omega \neq 0$ for all $i \in [m] \setminus Q$. By Condition \eqref{eq:condition3}, $y'_i=0$, hence $z'_{ti} = 0$ for all $i\in [m]\setminus Q$ and $t \in [m]$. Let $j\in Q$ and $i\in [m]\setminus \Gamma_j$. According to \eqref{eq:th4}, $\alpha_i < \sum_{r=1}^m \lambda_r \sigma_{ir} d_{ij}$, for being $(\alpha,\beta,\omega,\sigma)$ feasible, $\beta_{ij} \geq 0$. Hence, $\alpha_i - \beta_{ij} - \sum_{r=1}^m \lambda_r \sigma_{ir} d_{ij} \neq 0$ so  $z'_{ij} = 0$ by \eqref{eq:condition2}. Finally, with all of this it is clear that $y'_j = 1$ and $z'_{ij}=1$ for all $j\in Q$ and $i\in \Gamma_j$, resulting in $y' = \bar{y}$ y $z' = \bar{z}$ as was desired.
\end{proof}
Following the intuition behind the ordered contribution function and the above result, problem~\ref{eq:domp} will be LP-recovered if we can choose feasible dual variables $\alpha \in \mathbb{R}^m$ and $\sigma \in \mathbb{R}^{m \times m}$ satisfying:
\begin{enumerate}
    \item the overall ordered contribution to each center is the same, 
    \item the overall ordered contribution to any non-center point is smaller than that to a center point, and 
    \item each point \emph{sees} exactly its own center, i.e., $\alpha_i > \sum_{r=1}^m \lambda_r \sigma_{ir} d_{ij}$ if and only if $a_i$ is assigned to $a_j$.
\end{enumerate}
The following straightforward observation allows us to extract conditions for the simpler $1$-facility problem:
\begin{remark}[Single Center]
    Notice that in the case $p=1$, conditions~\eqref{eq:th1} and~\eqref{eq:th4} do not apply. Hence, if $j \in [m]$ is the only center, the conditions of Theorem~\ref{th:ocf} reduce to:
    \begin{align}
        & C_\sigma^\alpha(i) \leq C_\sigma^\alpha(j), && \forall i \in [m] \setminus \{j\}, \label{eq:singleth1}\\
        & \alpha_i \geq \sum_{r=1}^m \lambda_r \sigma_{ir} d_{ij}, && \forall i \in [m]. \label{eq:singleth2}
    \end{align}
\end{remark}

\subsection{Strongly Sortable Solutions}

Even though we interpret $\sigma$ as a permutation matrix that sorts the distance vector to the centers, the condition for recovering integrality stated in Theorem~\ref{th:ocf} refers to the existence of a matrix $\sigma$ that may, nevertheless, take non-integer values. In such a case, one may recover the location–allocation decision but not the ordering of distances. Below, we introduce some definitions and results that allow us to recover optimal values for the entire set of decision variables in the DOMP problem.
\begin{definition}[Strongly Sortable Solutions]
    Let $(\bar{y}, \bar{z}, u, v)$ be a mixed-integer optimal solution to \ref{eq:lp-bep}. This solution is said to be \emph{strongly sortable} if the dual problem~\ref{eq:lp-d} admits an optimal solution $(\alpha,\beta,\omega,\bar{\sigma})$ where $\bar{\sigma}$ takes binary values. In such a case $\bar{\sigma}$ can be seen as a permutation in $\mathscr{S}_m$ which satisfies $d_{\bar\sigma(r)}\geq d_{\bar\sigma(r+1)}$ for $r\in [m-1]$.
\end{definition}

The following result shows that, given a strongly sortable solution to \ref{eq:domp}, any permutation that sorts the distance vector in non-increasing order can be recovered by the dual LP relaxation \ref{eq:lp-d}.

\begin{lemma}\label{lemma:perm}
	Let $(\bar{y},\bar{z})$ be a strongly sortable solution to \ref{eq:lp-bep} such that $(\alpha,\beta,\omega,\bar{\sigma})$ is an optimal solution to \ref{eq:lp-d} where $\bar{\sigma}\in \mathscr{S}_m$. Let $\sigma \in \mathscr{S}_m$ any permutation fulfilling $d_{\sigma(r)}\geq d_{\sigma(r+1)}$ for $r\in [m-1]$. Then, there exists $\alpha' \in \R^m$ such that $(\alpha',\beta,\omega,\sigma)$ is an optimal solution to \ref{eq:lp-d}.
\end{lemma}

\begin{proof}
    If $d_{\bar{\sigma}(r)}> d_{\bar{\sigma}(r+1)}$ for $r\in [m-1]$ except in the centers (whose associated distance is zero), then the permutation $\bar{\sigma}$ is unique, but symmetry in centers indices. Let us say there are points $i,j\in [m]$ such that $d_i=d_j$. We consider the permutation $\sigma\in \mathscr{S}_m$ such that $\sigma(r) := \bar \sigma(r), \forall r \neq i,j$, $\sigma(i) := \bar{\sigma}(j)$, and $\sigma(j) := \bar{\sigma}(i)$. This is, $\sigma$ is equal to $\bar{\sigma}$ except in $i$ and $j$. Since $d_i = d_j$, we have that $\sigma$ is optimal solution to \eqref{eq:convexom}.

	Conditions \eqref{eq:condition1} and \eqref{eq:condition3} do not involve either $\alpha$ or $\sigma$, so that they remain true for any change in these variables. We take $\alpha'_t := \alpha_t, \forall t \neq i,j$, 
    \begin{equation*}
        \alpha'_i := \alpha_i + (\lambda_j - \lambda_i) d_{i i\star} \;\text{ and }\;\alpha'_{j} = \alpha_{j} + (\lambda_i - \lambda_j) d_{j j^\star};
    \end{equation*}
    where $a_{i^\star},a_{j^\star}$ are the centers where the points $a_i,a_j$ are allocated respectively, i.e., $\bar{z}_{i i\star}=\bar{z}_{j j\star}=1$. In this way, we have that
    \begin{equation*}
        \beta_{i i\star} = \alpha_i - \lambda_i d_{i i\star} = \alpha'_i - (\lambda_j - \lambda_i) d_{i i\star} - \lambda_i d_{i i\star} = \alpha'_i - \lambda_j d_{i i\star},
    \end{equation*}
    \begin{equation*}
        \beta_{j j\star} = \alpha_{j} - \lambda_j d_{j j\star} = \alpha'_{j} - (\lambda_i - \lambda_j) d_{j j\star} - \lambda_j d_{j j\star} = \alpha'_{j} - \lambda_i d_{j j\star}.
    \end{equation*}
	Hence, Condition \eqref{eq:condition2} is verified, and so $(\alpha',\beta,\omega,\sigma)$ is optimal solution to \ref{eq:lp-d}.

    In sum, that shows \ref{eq:lp-d} is invariant under the product of transpositions which respect the order, as was to be proved.
\end{proof}

Note that the strong condition above is related to an optimal solution to the pair \ref{eq:lp-bep} and \ref{eq:lp-d}. Thus, its verification can be difficult before solving the problem. In what follows, we provide a geometrical sufficient condition for the input data that assures the strength of a solution in terms of the data points.

\begin{proposition}\label{prop:equidistance}
	Let ${\rm D}$ be a distance matrix such that all its entries above the main diagonal are pairwise different. Then, if $(\bar{y}, \bar{z}, u, v)$ is a mixed-integer optimal solution to \ref{eq:lp-bep}, the solution is strongly sortable.
\end{proposition}

\begin{proof}
    Let $(\alpha,\beta,\omega,\sigma)$ be an optimal solution to \ref{eq:lp-d}. Let $\bar{\sigma}\in \mathscr{S}_m$ be the permutation that sorts the distance vector with respect to the centers $\{j\in [m]: \bar{y}_j=1\}$. Whether $\sigma \neq \bar{\sigma}$, we can assume $i,r \in [m]$ where $r \neq \bar{\sigma}^{-1}(i)$ and $\sigma_{ir}>0$. Nevertheless, the coefficient matrix of the system of linear equations \eqref{eq:tum1}-\eqref{eq:tu2} is totally unimodular, so $\sigma$ belongs to the face of optimal solutions to \eqref{eq:convexom} where $\bar{\sigma}$ must be a vertex. Such a face has, at least, dimension one, so there must exist another vertex $\sigma'\in \mathscr{S}_m$ optimum to \eqref{eq:convexom} where $\sigma'_{ir}=1$. Hence, $\lambda_{\bar{\sigma}^{-1}(i)}=\lambda_r$, for all $i, r\in [m]$ such that $\sigma_{ir}>0$. Eventually, we have that $\sum_{r=1}^m \lambda_r \sigma_{ir} = \lambda_{\bar{\sigma}^{-1}(i)} \sum_{r=1}^m \sigma_{ir} = \lambda_{\bar{\sigma}^{-1}(i)}$, thus by conditions \eqref{eq:condition1}-\eqref{eq:condition3}, $(\alpha,\beta,\omega,\bar{\sigma})$ is optimal solution to \ref{eq:lp-d}.
\end{proof}

For the sake of simplicity, the condition required in the above result will be referred to as the \emph{free-of-equidistance} condition for the input points, which we define as follows:
\begin{definition}[Free-of-Equidistance Points]
The set of points $P$ is said to be \emph{free-of-equidistance} if all entries above the main diagonal of its distance matrix are pairwise distinct.
\end{definition}
Note that this condition, that can be easily checked with the input data, means that there are no points at the same distance from each other. Given a set of centers $Q$, then $d_i\neq d_{j}$ for all $i\neq j\in [m]\setminus Q$, so there is a unique permutation that sorts the distance vector from non-centers to the centers in decreasing order.\\

\subsection{Conic Combination of ${\bf \lambda}$--Weights}

We analyze how to extend the results obtained for a given set of $\bm{\lambda}$ vectors to their combinations. This occurs, for instance, when combining the $p$-median and $p$-center problems in the so-called $p$-$\gamma$-centdian problem.

Given the family of vectors $\Lambda_i:=(\lambda_j^i)_{j\in [m]}\in \R^m$ defined by $\lambda_j^i=1$ if $j\leq i$; $0$ otherwise for all $i\in [m]$. It is easy to see that the set of possible lambda weights for a convex \ref{eq:domp} with nonnegative weights is the conic hull, $\cone\left(\Lambda_1,\ldots, \Lambda_m\right)$, of $\Lambda$'s. Furthermore, if \ref{eq:lp-bep} recovers integrality for weights $\bm{\lambda}$, then it also recovers integrality for weights $c\bm{\lambda}, \forall c\in \R_+$, and all of them take the same optimal values in the integral variables $y$ and $z$. Hence, the study of integrality recovery in \ref{eq:domp} is invariant along rays in $\cone\left(\Lambda_1,\ldots, \Lambda_m\right)$; that is why we can consider the convex hull, $\conv\left(\Lambda_1,\ldots, \Lambda_m\right)$, of $\Lambda$'s as the set of canonical representatives of discrete ordered $p$-median problem classes over the same set of input points $P$.

\begin{proposition} \label{prop:weights}
	Let \ref{eq:lp-bep} and $p\text{-}{\rm DOMP}^{\bm{\mu}}_{\rm LP}(\D)$ be two {\rm LP} relaxations. If both problems have the same strongly sortable solution, then the problem $p\text{-}{\rm DOMP}^{\bm{\lambda}+\bm{\mu}}_{\rm LP}(\D)$ also has the same strongly sortable solution.
\end{proposition}

\begin{proof}
    By Theorem~\ref{th:ocf} and Lemma~\ref{lemma:perm}, there exist $\alpha^{\bm{\lambda}}, \alpha^{\bm{\mu}} \in \R^m$, and the same $\sigma \in \mathscr{S}_m$ for both, such that \eqref{eq:th1}-\eqref{eq:th4} hold. Making an abuse of notation, hereinafter this proof, $C^\alpha_{\bm{\lambda}}:=C^{\alpha^{\bm{\lambda}}}_\sigma$, $C^\alpha_{\bm{\mu}}:=C^{\alpha^{\bm{\mu}}}_\sigma$, and $C^\alpha_{{\bm{\lambda}}+{\bm{\mu}}}:=C^{\alpha^{\bm{\lambda}}+\alpha^{\bm{\mu}}}_\sigma$. It is direct $\alpha^{\bm{\lambda}}+\alpha^{\bm{\mu}}$ and $\sigma$ satisfy \eqref{eq:th3} and \eqref{eq:th4} for weights ${\bm{\lambda}}+{\bm{\mu}}$. By subadditivity of the positive part function, we have that $C^\alpha_{{\bm{\lambda}}+{\bm{\mu}}}(j)\leq C^\alpha_{{\bm{\lambda}}}(j)+C^\alpha_{{\bm{\mu}}}(j)$ for all $j\in [m]$. Moreover, this inequality becomes tight when both addends in $\bullet_+$ have the same sign. Conditions~\eqref{eq:th3} and \eqref{eq:th4} state that such addends have the same sign for centers in the problems with weights ${\bm{\lambda}}$ and ${\bm{\mu}}$; therefore, the inequality is an equality for those, so \eqref{eq:th1} is satisfied. Finally, if $a_i$ is not a center, $C^\alpha_{{\bm{\lambda}}+{\bm{\mu}}}(i)\leq C^\alpha_{{\bm{\lambda}}}(i)+C^\alpha_{{\bm{\mu}}}(i)\leq C^\alpha_{{\bm{\lambda}}}(j)+C^\alpha_{{\bm{\mu}}}(j)=C^\alpha_{{\bm{\lambda}}+{\bm{\mu}}}(j)$ for all center $a_j$, verifying \eqref{eq:th2}. By Theorem~\ref{th:ocf}, the problem $p\text{-}{\rm DOMP}^{\bm{\lambda}+\bm{\mu}}_{\rm LP}(\D)$ recovers integrality, giving the same optimal solution as the former problems. Besides, it is easy to check the strength of this solution.
\end{proof}

As a consequence of Proposition~\ref{prop:weights} and the comments above, we know that if two \ref{eq:domp} with different weights share a strongly sortable solution, it will also be a strongly sortable solution to the \ref{eq:domp} with the conic combination of those weights. Nevertheless, this is a sufficient condition for recovery, but not necessary; there are conic combinations of weights which allow recovery, whereas their extremes do not.\\

\subsection{The Single--Center Case}

Below, we will show some results in the single-center case where there is no combinatorics in the allocation, so retrieving the question in the air: can we find the center via the LP relaxation? The characterization of this simpler case is the key to the subsequent analysis of the concrete problems.

The simplest case of the DOMP corresponds to the single-center setting ($p=1$). In this case, there is no need to solve a mathematical optimization problem to obtain a solution, since it can be derived directly by evaluating the ordered median operator in every singleton, i.e., ${\rm OM}_{\bm{\lambda}}(\{a_j\})$ for all $j \in [m]$. This requires only sorting the values in the $j$th column of the distance matrix and aggregating them with the appropriate ${\bm{\lambda}}$, a task that can be performed efficiently by ``naive" methods.

In what follows, we analyze the implications for the primal–dual characterization given in Theorem~\ref{th:ocf} for this case, and how recovery extends from the single-center to the multi-center setting. Specifically, the next result establishes necessary conditions for \ref{eq:domp} to recover integrality, based on the local $p=1$ version of the problem.

\begin{proposition} \label{prop:decomposition}
	Let $(P,\bm{\lambda},p)$ with $p>1$ be the input data of \ref{eq:domp} such that it recovers integrality with a strongly sortable solution. Let $\Gamma \subseteq [m]$ be all indices allocated to a center $a_j$, and $\sigma \in \mathscr{S}_m$ be a permutation which sorts the distance vector $(d_i)_{i\in [m]}$. The \ref{eq:domp} with input data $(P|_\Gamma,\bm{\lambda}_\Gamma,1)$ also recovers integrality and has $a_j$ as only center, where $P_\Gamma:=\{a_i\in P : i\in \Gamma\}$ and $\bm{\lambda}_\Gamma:=\left(\lambda_r\right)_{\sigma(r)\in \Gamma} \in \R^{|\Gamma|}$.
\end{proposition}

\begin{proof}
    Let $(\bar{y},\bar{z},u,v)$ and $(\alpha, \beta, \omega, \sigma)$ be strongly sortable solutions for the primal-dual pair \ref{eq:lp-bep} and \ref{eq:lp-d}. It is easy to check conditions \eqref{eq:condition1}-\eqref{eq:condition3} still hold if $i$ and $j$ range over $\Gamma$ and $r$ ranges over $\sigma^{-1}(\Gamma)$.
\end{proof}

The above result implies that the multi-center problem has a worse recovery performance than the single-center case. Thus, there would be no hope to recover integrality for the problem if its $p=1$ does not recover.

The following result provides necessary and sufficient conditions for the single-center case, via the ordered contribution function, to recover integrality.

\begin{theorem} \label{th:characterization}
    Let $P\subset \R^n$ be a free-of-equidistance set of points. Let $a_j \in P$ be a point and $\sigma \in \mathscr{S}_m$ be a permutation sorting the distances to $a_j$. Then, $1$-{\rm DOMP} recovers integrality with $a_j$ as center if and only if
    \begin{equation*}
        \sum_{t=1}^m \lambda_{\sigma^{-1}(t)} d_{tj} \leq \sum_{t=1}^m \lambda_{\sigma^{-1}(t)} d_{ti}, \; \forall i \in [m].
    \end{equation*}
\end{theorem}
\begin{proof}
    The sufficient condition is straightforward. One can just take the vector $\alpha$ so that $\alpha_t \geq \lambda_1 \max_{i,j \in [m]} d_{ij}$, i.e., at least $\lambda_1$ (the largest weight) times the farthest distance between points in $P$, componentwise. Using the inequalities of the hypothesis, it is easy to see that \eqref{eq:singleth1} and \eqref{eq:singleth2} hold; thus, Theorem~\ref{th:ocf} ensures that the solution is optimal to $1\text{-}{\rm DOMP}^{\bm{\lambda}}_{\rm LP}(\D)$.

    Conversely, it remains to prove the necessary condition. For fixed $i,t \in [m]$, whether $\alpha_t < \lambda_{\sigma^{-1}(t)} d_{ti}$ we have that
\begin{equation*}
    \alpha_t - \lambda_{\sigma^{-1}(t)} d_{tj} - (\alpha_t - \lambda_{\sigma^{-1}(t)} d_{ti})_+ = \alpha_t - \lambda_{\sigma^{-1}(t)} d_{tj} < \lambda_{\sigma^{-1}(t)} d_{ti} - \lambda_{\sigma^{-1}(t)} d_{tj}.
\end{equation*}
Hence, we always have
\begin{equation}\label{eq:th9}
    \alpha_t - \lambda_{\sigma^{-1}(t)} d_{tj} - (\alpha_t - \lambda_{\sigma^{-1}(t)} d_{ti})_+ \leq \lambda_{\sigma^{-1}(t)} d_{ti} - \lambda_{\sigma^{-1}(t)} d_{tj}.
\end{equation}

Theorem~\ref{th:ocf} ensures \eqref{eq:singleth1} and \eqref{eq:singleth2} hold for some $\alpha$. Thus, for all $i\in [m]$, it follows 
    \begin{align*}
        0\leq C_{\sigma}^\alpha(j) - C_{\sigma}^\alpha(i) &= \sum_{t=1}^m \left[ \alpha_t - \lambda_{\sigma^{-1}(t)} d_{tj} - (\alpha_t - \lambda_{\sigma^{-1}(t)} d_{ti})_+ \right]\\
        & \leq \sum_{t=1}^m \lambda_{\sigma^{-1}(t)} d_{ti} - \sum_{t=1}^m\lambda_{\sigma^{-1}(t)} d_{tj},
    \end{align*}
where the former inequality comes from \eqref{eq:singleth1}, the equality comes from \eqref{eq:singleth2}, and the latter inequality comes from \eqref{eq:th9}.
\end{proof}

\section{On Some Special Problems}\label{sec:specialcases}

In this section, we study some particular problems of interest, namely those specified in Table \ref{tab:cases}. 
\begin{table}[h!]
\centering
\begin{tabular}{c|c|c}
    $\mathbf{\lambda}$ & Problem & Acronym \\ \hline
    $(1, \ldots, 1)$ & $p$-median & \kM \\
    $(1, 0, \ldots, 0)$ & $p$-center & \kC \\
    $(\underbrace{1, \ldots, 1}_{k}, 0, \ldots, 0)$ & $p$-$k$-sum & \kS \\
    $(1, \gamma, \ldots, \gamma)$ & $p$-$\gamma$-centdian & \kD
\end{tabular}
\caption{Special problems of interest analyzed in Section~\ref{sec:specialcases} and their acronyms used in Section~\ref{sec:exps}.}
\label{tab:cases}
\end{table}

The first case we study corresponds to the most popular problem within the DOMP family: the $p$-median problem. The conditions on the input distances required to recover integrality in this problem have been extensively analyzed in recent works~\citep[see, e.g.,][]{r:del2023k}. Several favorable, and even optimistic, properties have been established for this case, contributing to a rather clear understanding of its structure and behavior. However, when we shift our attention to ordered problems, where the combinatorial complexity increases, the situation changes substantially. In such cases, the same properties cannot be expected to hold, as the additional elements that define problems such as the $p$-center or the $p$-$\gamma$-centdian introduce new layers of complexity. Our findings highlight the intrinsic difficulty of addressing this richer setting, where the interaction among multiple factors prevents the preservation of the desirable characteristics observed for the $p$-median problem. Moreover, when considering combinations of ordered median operators to explore the differences and similarities between DOMPs defined by distinct weights, we conjecture that problems exhibit better behavior, in terms of the tightness of their LP relaxation, when they are closer to the $p$-median, and worse behavior as they move away from it, particularly when robustness considerations dominate the decision-making process.

The methodology we adopt builds directly on the structure of the location–allocation decision variables that define the problem. On the one hand, it is easy to observe that the behavior of the allocation variables in their relaxed form improves whenever the distribution of the input points can be naturally partitioned into $p$ clusters. In fact, the most favorable situation arises when each point within a cluster is closer to every other point in that same cluster (within-cluster distance) than to any point in a different cluster (between-cluster distance). In this idealized setting, integrality recovery for the allocation variables is guaranteed, and the remaining question is whether the centers of the clusters can also be recovered through the relaxed location variables.

Prior work on this topic for the classical $p$-median problem, has focused on quantifying how tightly the input point distribution can be clustered while still ensuring that the LP relaxation of the problem remains exact with high probability. Following this line, we show that once sorting constraints are introduced into the problem, the tightness of the relaxation is lost with respect to the location variables. Consequently, even under the most favorable input distributions, one cannot expect the LP relaxation to retain the same desirable properties.

\begin{theorem}\label{th:specialcases}
    Let $P\subset \R^n$ be a free-of-equidistance set of points and let $p$ be a natural. Then, the following statements hold.
    \begin{enumerate}
        \item If $\sum_{r=2}^m \lambda_r < \lambda_1$, the \ref{eq:domp} does not recover integrality.\label{prop:one}
        \item If no three points are collinear and $\sum_{r=2}^m \lambda_r = \lambda_1$, the \ref{eq:domp} does not recover integrality.\label{prop:two}
        \item ${\rm M}_1$ recovers integrality (even of the points are non free of equidistance). \label{prop:median}
        \item If no three points are collinear, ${\rm S}_p^2$ does not recover integrality.\label{prop:sum}
        \item ${\rm C}_p$ does not recover integrality.\label{prop:center}
        \item Let $0\leq\gamma\leq1$ be a real. If $\gamma(m-1)< 1$, the ${\rm D}_p^\gamma$ does not recover integrality.\label{prop:centdian}
    \end{enumerate}
\end{theorem}
\begin{proof}
    Firstly, we show \ref{prop:one}. Let us assume that the $1$-{\rm DOMP} recovers integrality with a strongly sortable solution with $a_j$ as center and $i_1 := \sigma(1)$, where $\sigma \in \mathscr{S}_m$ is a permutation that sorts the distances to $a_j$. We have that
\begin{align*}
    0 \leq \sum_{t=1}^m \lambda_{\sigma^{-1}(t)} (d_{ti_1} - d_{tj}) & = - \lambda_1 d_{i_1j} + \sum_{\stackrel{t=1}{t\neq i_1}}^m \lambda_{\sigma^{-1}(t)} (d_{ti_1} - d_{tj}) \\
    & \leq - \lambda_1 d_{i_1j} + d_{i_1j}\sum_{r=2}^m \lambda_r,
\end{align*}
where the first inequality comes from Theorem~\ref{th:characterization} and the last one comes from the triangular inequality $d_{tj} + d_{i_1j} \geq d_{ti_1}$ for all $t \in [m] \setminus \left\{ i_1 \right\}$. Thus $- \lambda_1 + \sum_{r=2}^m \lambda_r \geq 0$. Now, we have to apply Proposition~\ref{prop:decomposition} to prove the claim for general \ref{eq:domp}.

To show \ref{prop:two}, we consider again that $1$-{\rm DOMP} recovers integrality and the equality in the hypothesis holds. Then 
\begin{equation*}
    \sum_{\stackrel{t=1}{t\neq i_1}}^m \lambda_{\sigma^{-1}(t)} (d_{ti_1} - d_{tj}) = \lambda_1 d_{i_1j},
\end{equation*}
what implies $d_{ti_1} - d_{tj}=d_{i_1j}$ for every $t\in [m]$ such that $t\neq i_1$ and $\lambda_{\sigma^{-1}(t)}\neq 0$. Hence, the final claim follows directly from Proposition~\ref{prop:decomposition} and \ref{prop:one}.

To continue, we show \ref{prop:median}. Let $a_j\in P$ be the median of the solution to \ref{eq:milp-bep}. We take $\alpha\in \R^m$ such that $\alpha_t \geq \max_{i,j \in [m]} d_{ij}$ for every $t \in [m]$, i.e., each component of $\alpha$ is at least the farthest distance between points of $P$. We also take $\sigma \in \mathscr{S}_m$ any permutation sorting the vector of distances to $a_j$. In such a way, $\alpha_i \geq d_{ij} = \sum_{r=1}^m \lambda_r \sigma_{ir} d_{ij}, \forall i \in [m]$, so \eqref{eq:singleth2} is satisfied.

    On the other hand, for every $i\in [m]$
    \begin{equation*}
        C_\sigma^\alpha(i) = \sum_{t=1}^m \left( \alpha_t - \sum_{r=1}^m \lambda_r \sigma_{tr} d_{ti} \right)_+ = \sum_{t=1}^m \alpha_t - \sum_{t=1}^m d_{ti}.
    \end{equation*}
	It is clear that $j$ minimizes $i \mapsto \sum_{t=1}^m d_{ti}$, then, maximizes $C_\sigma^\alpha(i)$, satisfying \eqref{eq:singleth1}. By Theorem~\ref{th:ocf}, the $1$-median recovers integrality with $a_j$ as center.

    Finally, \ref{prop:sum} follows directly from \ref {prop:two}; and  \ref{prop:center} and \ref{prop:centdian} are consequence of \ref{prop:one}.
\end{proof}

Note that, although Proposition~\ref{prop:median} does not hold in general for $p > 1$, if the input points are distributed into $p$ clusters that are pairwise sufficiently far apart, the LP relaxation of the $p$-median problem remains tight. This, however, is no longer the case for the $p$-2-sum (Proposition~\ref{prop:sum}) or $p$-center (Proposition~\ref{prop:center}) problems. In these settings, even when the input data is free of equidistance, not even the most favorable clustered distribution of the input points yields a tight LP relaxation. Finally, Proposition~\ref{prop:centdian} shows that $m \geq \frac{1}{\gamma} + 1$ is a necessary condition for integrality recovery in the $p$-$\gamma$-centdian problem. This implies that, in order to observe any improvement in the performance of the LP relaxation, either the $p$-$\gamma$-centdian must behave similarly to the $p$-median (i.e., for values of $\gamma$ close to one) or the number of points must be large ($m$ large). Asymptotically, this condition indicates that for the $p$-center problem (where $\gamma = 0$), there is no possibility for the LP relaxation to be tight.

\section{Empirical Results}\label{sec:exps}

This section presents the results of a series of computational experiments designed to evaluate the effect of sorting when recovering solutions of different DOMP problems from their linear relaxations. The benchmark datasets used correspond to the first $20$ \texttt{pmed} instances (\texttt{pmed1}–\texttt{pmed20}) from ORLIB. The input cardinality $m$ takes values in $\{100, 200, 300, 400\}$, with instances \texttt{pmed1–5}, \texttt{pmed6–10}, \texttt{pmed11–15}, and \texttt{pmed16–20} corresponding to each value, respectively. 

It is worth mentioning that our goal is not to provide a standard computational study as commonly found in the literature on the DOMP, which typically aims to identify the limitations of MILP formulations~\citep[see, e.g.]{Marin2020,Labbe2017}. Instead, we focus on analyzing the quality of the LP relaxations and the features that most influence these bounds. Specifically, we first examine the differences among the possible variants of the DOMP considered (in light of the theoretical results established in the previous sections), and subsequently analyze the impact of the input data structure, whose characteristics will be measured in terms of their clusterability.

As already observed in recent works~\citep[see, e.g.,][]{awasthi2015relax,r:nellore2015recovery,r:del2023k}, the integrality recovery properties of an oracle are strongly influenced by the distributional characteristics of the input points. Specifically, for the $p$-median problem, the authors proved that when the input points are generated around $p$ Euclidean balls whose centers are pairwise sufficiently separated in space, the probability of recovering the integral solution from its relaxation is high.

In our case, the input data from ORLIB are not generated according to any geometric model, beyond a random generation of symmetric and with zero-diagonal distance matrices, where the triangle inequality holds, and the number of potential clusters is unknown. Therefore, for each instance, we compute a measure that allows us to identify whether it is \emph{clusterizable}, in the sense that several points can be grouped into coherent clusters; otherwise, the data are ``uniformly'' distributed in space. This information will help us empirically assess whether the spatial distribution of the points affects the LP relaxation of the problem, as has been observed in the particular case of the $p$-median problem (with $p$ clusters).

\subsection*{Experiment Design}

For each of the instances, we run the DOMP model for the the cases $p$-median, $p$-center, $p$-$\gamma$-centdian, and $p$-$k$-sum for $p \in \{2,3,5\}$, $\gamma \in \{\frac{1}{4},\frac{1}{2},\frac{3}{4}\}$, and $k \in \left\{2,\left\lceil \frac{1}{4} m\right\rceil, \left\lceil \frac{1}{2} m\right\rceil,\left\lceil \frac{3}{4} m\right\rceil\right\}$, where $m$ is the cardinality of the instance. In total, we solved $20 \times 3 (1 + 1 + 3 + 4) = 540$ instances of the DOMP. 

All computational experiments were performed on Huawei FusionServer Pro XH321 (\texttt{albaic\'in} at Universidad de Granada ( \url{https://supercomputacion.ugr.es/arquitecturas/albaicin/}) with an Intel Xeon Gold 6258R CPU @ 2.70GHz with 28 cores. Optimization tasks were solved using Gurobi Optimizer version~12.0.1 within a time limit of 2 hours.

For each solved instance, we collected information regarding both the computational effort required to solve the problem and the quality of the LP relaxation with respect to the optimal MILP solution. The following performance indicators were recorded: the CPU time required to optimally solve the instance (\texttt{CPUTime}); the MIP gap when the instance was not solved to optimality within the time limit (\texttt{MIPGap}); the number of nodes explored in the branch-and-bound tree until the problem was solved or the time limit was reached (\texttt{Nodes}); the best value of a feasible solution obtained within the time limit (\texttt{BestObj}); the best lower bound (LP relaxation) obtained within the time limit (\texttt{BestRel}); the gap of the LP relaxation with respect to the best feasible value (\texttt{GapLP}); and the gap of the LP relaxation at the end of exploring the root node with respect to the best feasible value (\texttt{GapLPRoot}).\\

\subsection*{Computational Performance}

In what follows, we analyze the computational performance of the models and their relation to the quality of the LP relaxation (and its impact on the required computational effort). 

A first overview of the obtained results, which will be justified by the deeper study presented later, is given by the performance profile by type of DOMP shown in Figure~\ref{fig:pp}. On the $x$-axis we represent the CPU times (in seconds), whereas on the $y$-axis we show the percentage of instances solved within each CPU time. One can observe that the $p$-median problems (M) are the least time-demanding, followed by the $p$-$k$-sum problems (S), then the $p$-$\gamma$-centdian problems (D), and finally the most challenging case are the $p$-center problems (C). This simple first description already provides an idea of the impact of the quality of the LP relaxation when solving the MILP problems under consideration. We proved, theoretically, that the $p$-center problem never recovers integrality and, consequently, requires more iterations to obtain an integral solution. However, for the $p$-median problem, based on the previous results, there is some hope to solve it by means of its LP relaxation, as is also the case for the $p$-$\gamma$-centdian and $p$-$k$-sum problems (for adequate values of $\gamma$ and $k$, respectively). Specifically, for those problems, Figure~\ref{fig:pp_DS} reports the performance profiles disaggregated by the parameter defining the problem ($\gamma$ or $k$), where it can be observed, more emphatically for the $p$-$k$-sum, that certain parameter values induce significantly more difficult problems than others.

\begin{figure}[h]
\centering
\includegraphics[width=.75\textwidth]{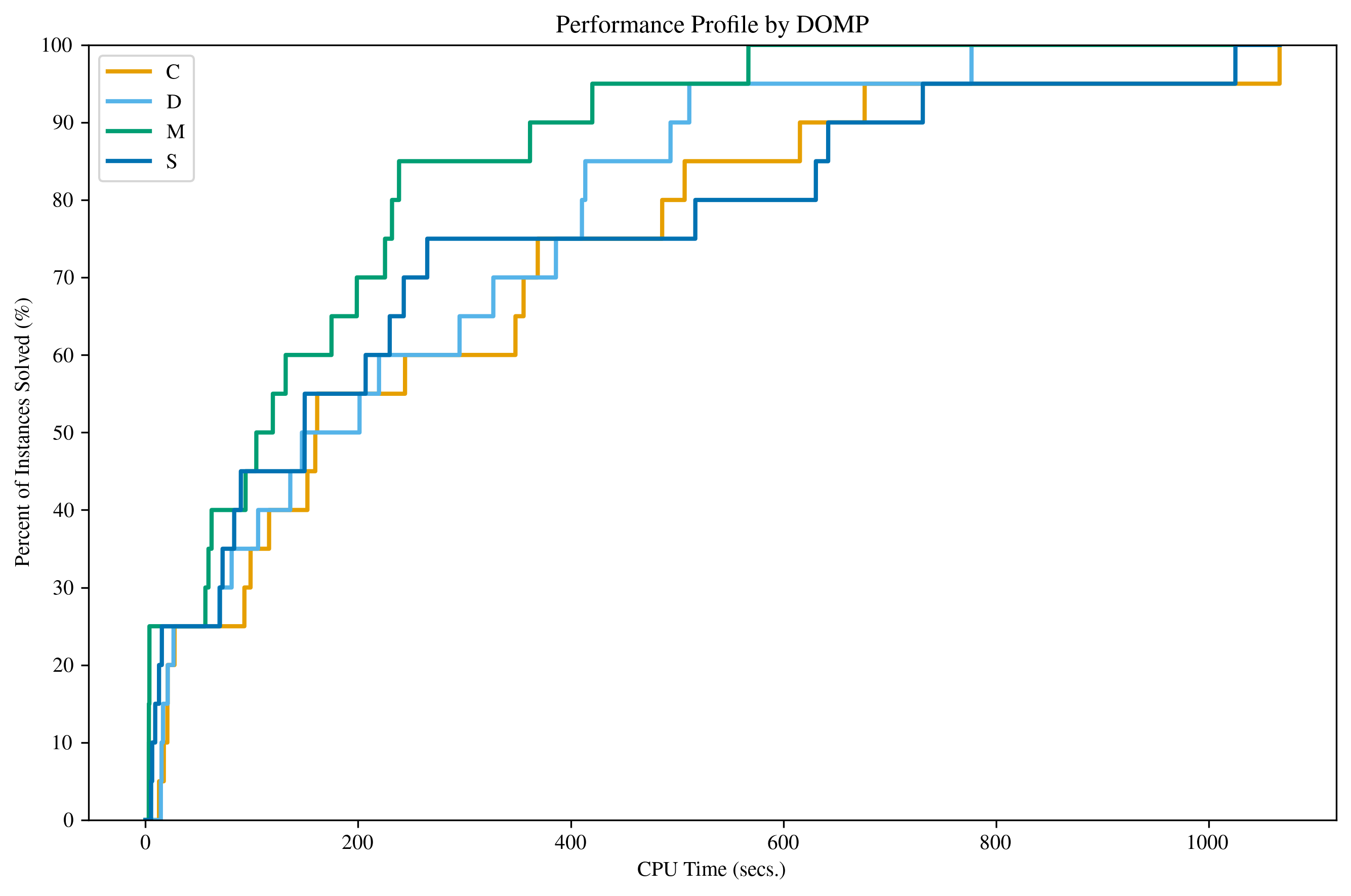}
\caption{Performance profile for all the instances classified by type of DOMP.\label{fig:pp}}
\end{figure}

\begin{figure}[h]
\centering
\adjustbox{width=\textwidth}{\includegraphics{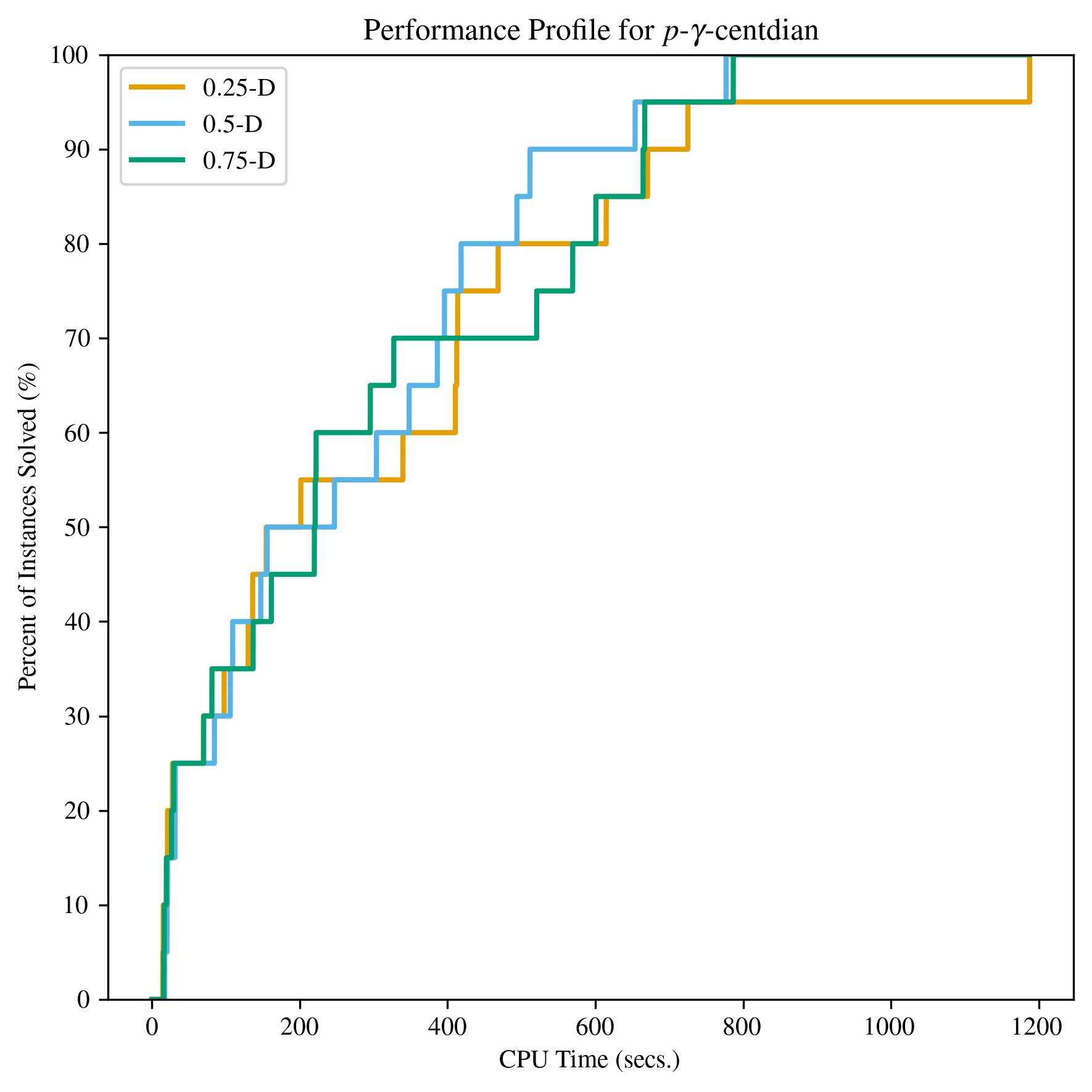}~\includegraphics{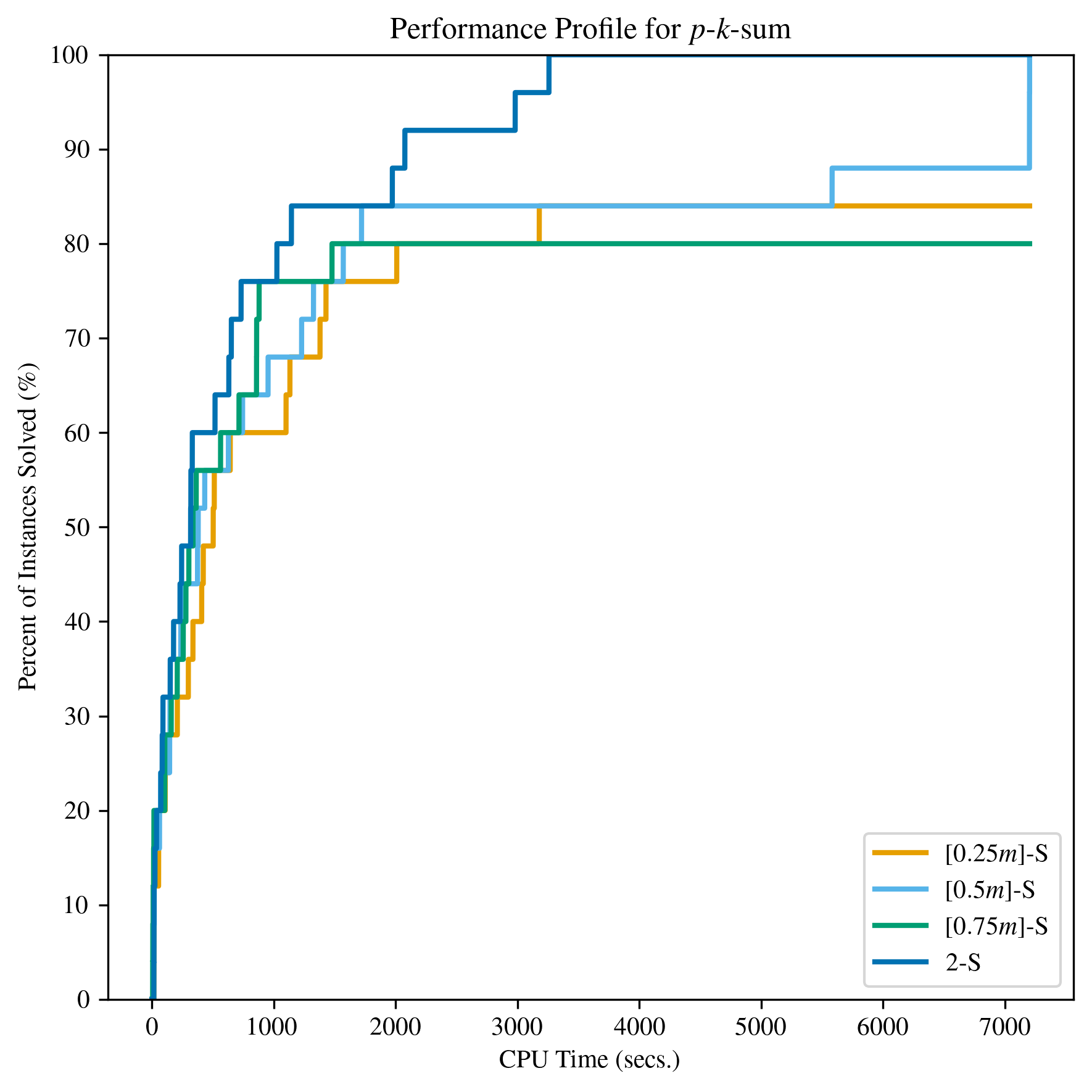}}
\caption{Performance profile for the different types of centdian (left) and sum (right) problems.\label{fig:pp_DS}}
\end{figure}

This information provides a first indication of how challenging it is to solve a sorted version of an optimization problem compared to its non-sorted counterpart in terms of computational effort. This also highlights why ordered optimization deserves further analysis to design computationally efficient solution algorithms. Specifically, one can observe that, although the $p$-median problem is already NP-hard, it is less demanding than other variants of the DOMP that explicitly require sorting the allocation distances. Among these, the $p$-center problem appears to require less CPU time than other DOMPs for instances solved within the first few minutes, but it becomes considerably more challenging to close its MIP gap compared to the centdian or sum problems. 

In particular, we focus on identifying when the LP relaxation of the problem constitutes an efficient exact approach for solving the problem in polynomial time. Thus, in the following analysis, we examine the impact of the quality of the LP relaxation in terms of its tightness on the optimal solution process. Specifically, we relate the computational effort required by a branch-and-bound algorithm for solving a MILP, in terms of the number of nodes explored in the search tree, to the LP gap of the problem.

In Figure~\ref{fig:lpgap_nodes}, we present scatter plots of the number of explored nodes until the best solution is found versus the LP gap of the problem, for each of the different types of DOMP analyzed in our experiments. We also include the exponential fit of the data (red curve in the plots), which clearly shows that the number of explored nodes is exponentially affected by the quality of the LP relaxation. Thus, even if the problem cannot be solved exactly through its LP relaxation, the closer the LP bound is to the true MILP solution, the fewer nodes are required to solve the problem using a branch-and-bound approach. Consequently, the problem becomes less computationally demanding in terms of both memory and running time.

\begin{figure}[h!]
\centering
\adjustbox{width=0.8\textwidth}{\includegraphics{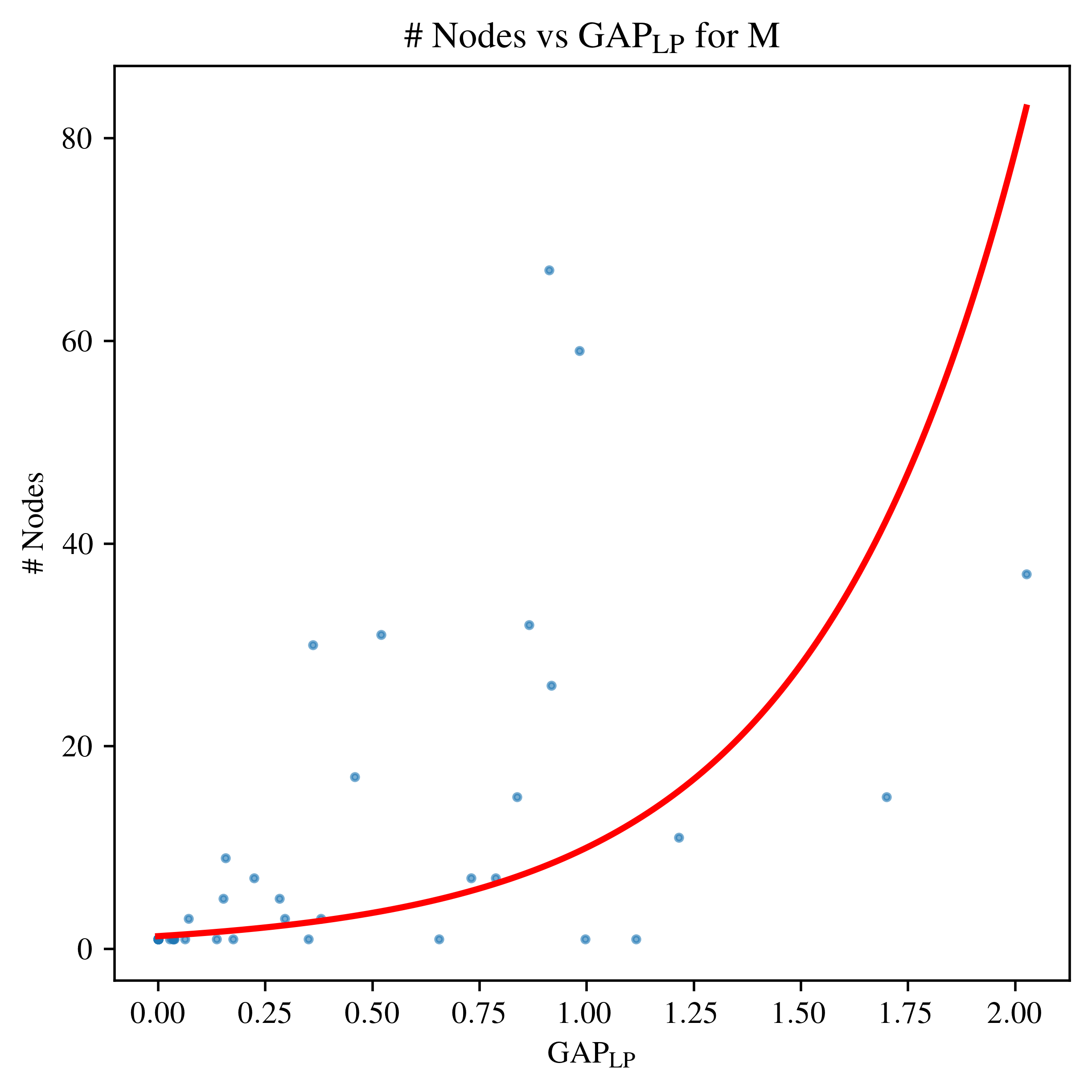}~\includegraphics{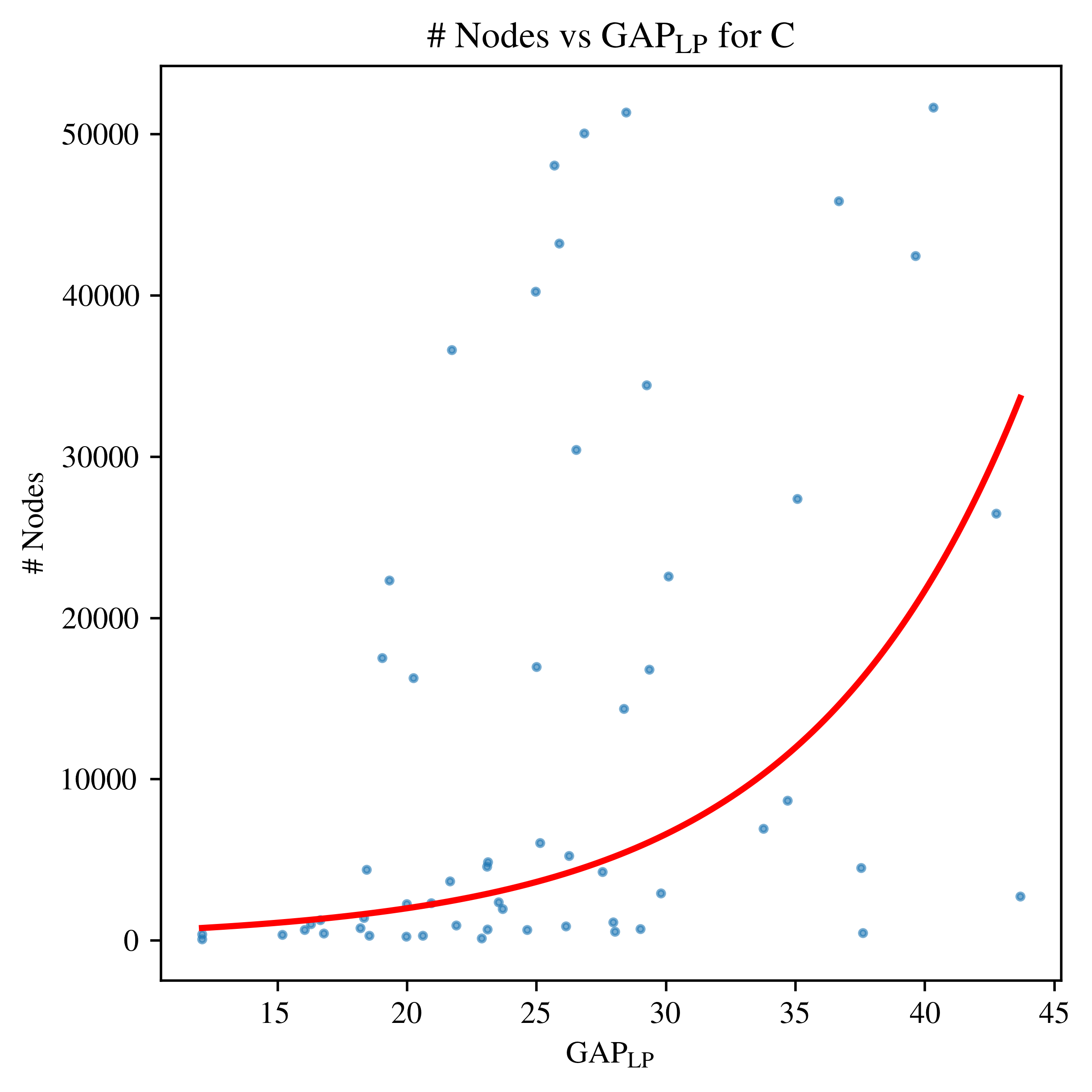}}\\
\adjustbox{width=0.8\textwidth}{\includegraphics{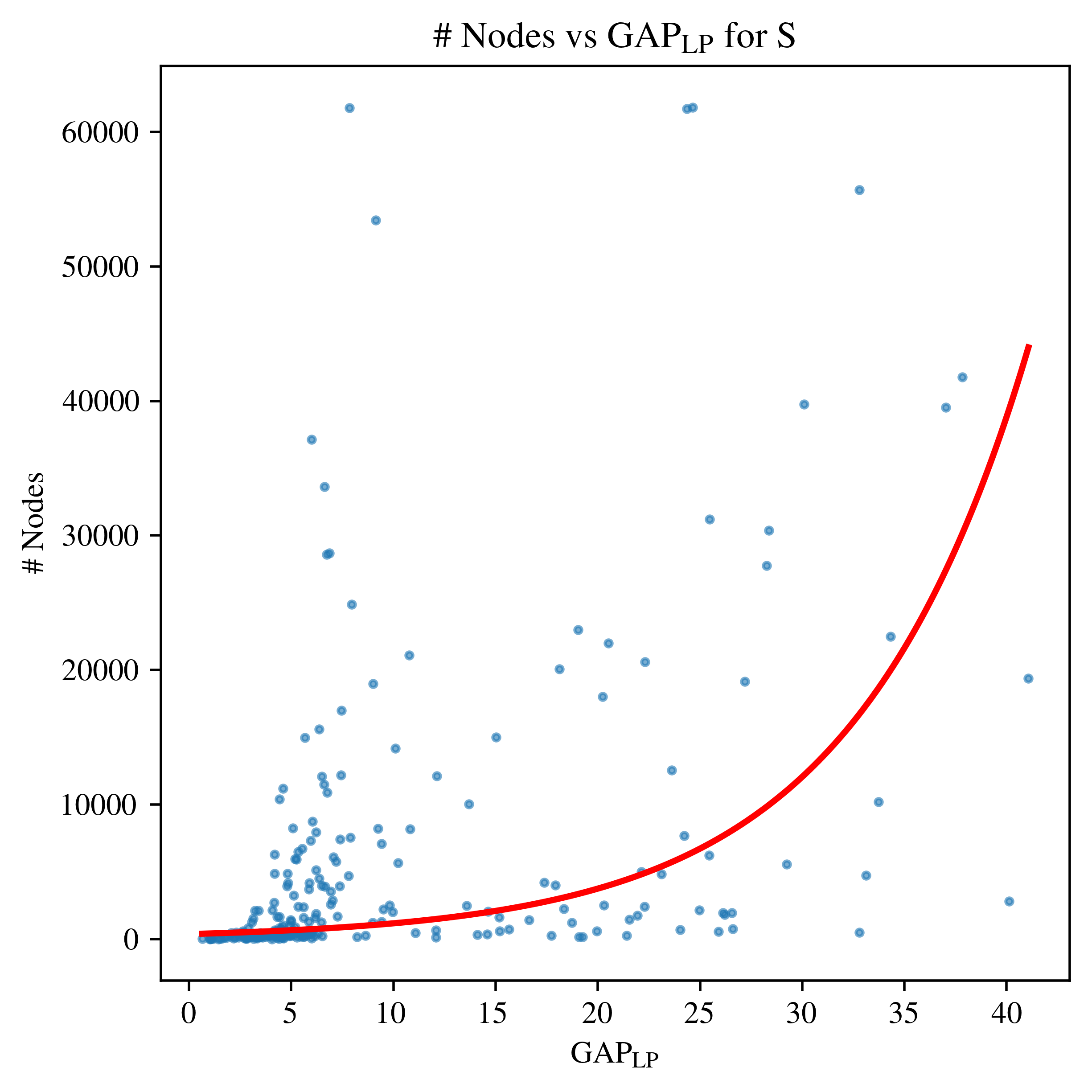}~\includegraphics{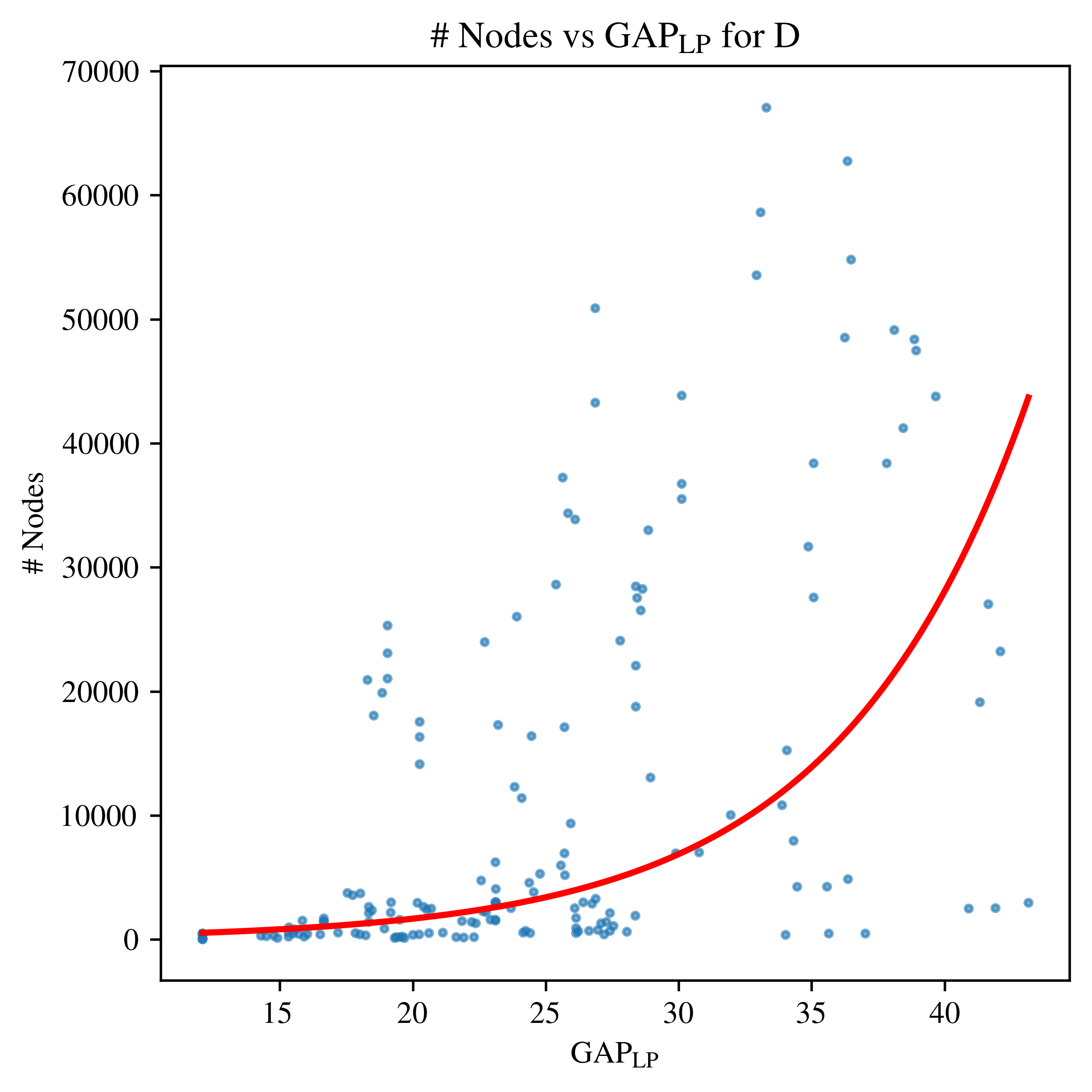}}
\caption{Scatter plots and exponential fits for the number of nodes with respect to the LP gap for each of the different types of DOMP  (from top left to bottom right: median, center, sum, and centdian). \label{fig:lpgap_nodes}}
\end{figure}

In order to detect what is influencing the computational demand required to solve the different types of problems, in what follows, we analyze the quality of the LP relaxation for all these instances. For each instance, we compute the LP gap (${\rm GAP}_{LP}$) as the relative deviation between the LP relaxation of the problem and the best solution obtained within the time limit. First, in Figure \ref{fig:lpgap_all}, we show for the four families of problems that we solved the LP gap performance profile, i.e., we represent the percent of instances that got an LP relaxation gap smaller than the threshold indicated in the $x$-axis. 

\begin{figure}[h!]
\centering
\adjustbox{width=0.8\textwidth}{\includegraphics{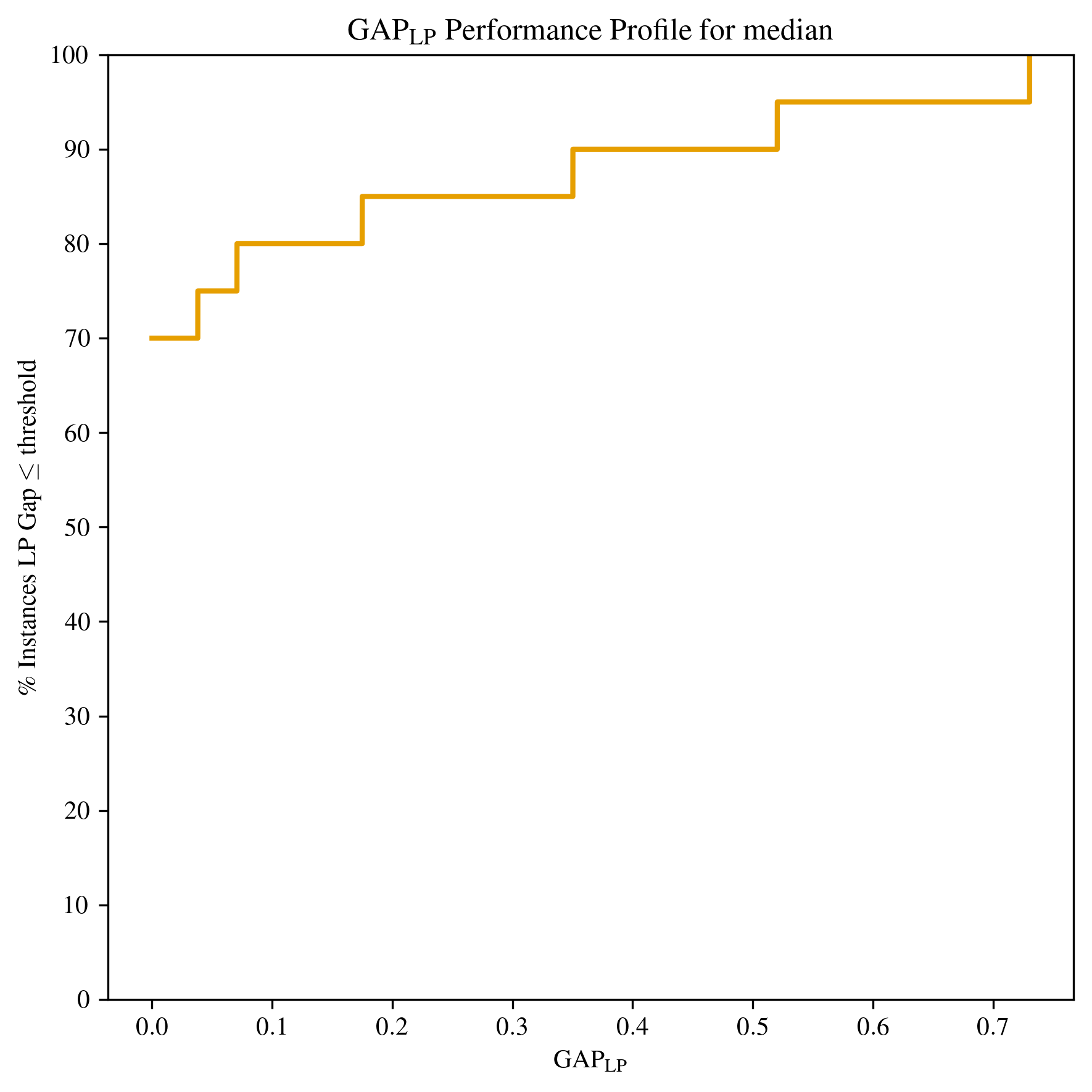}~\includegraphics{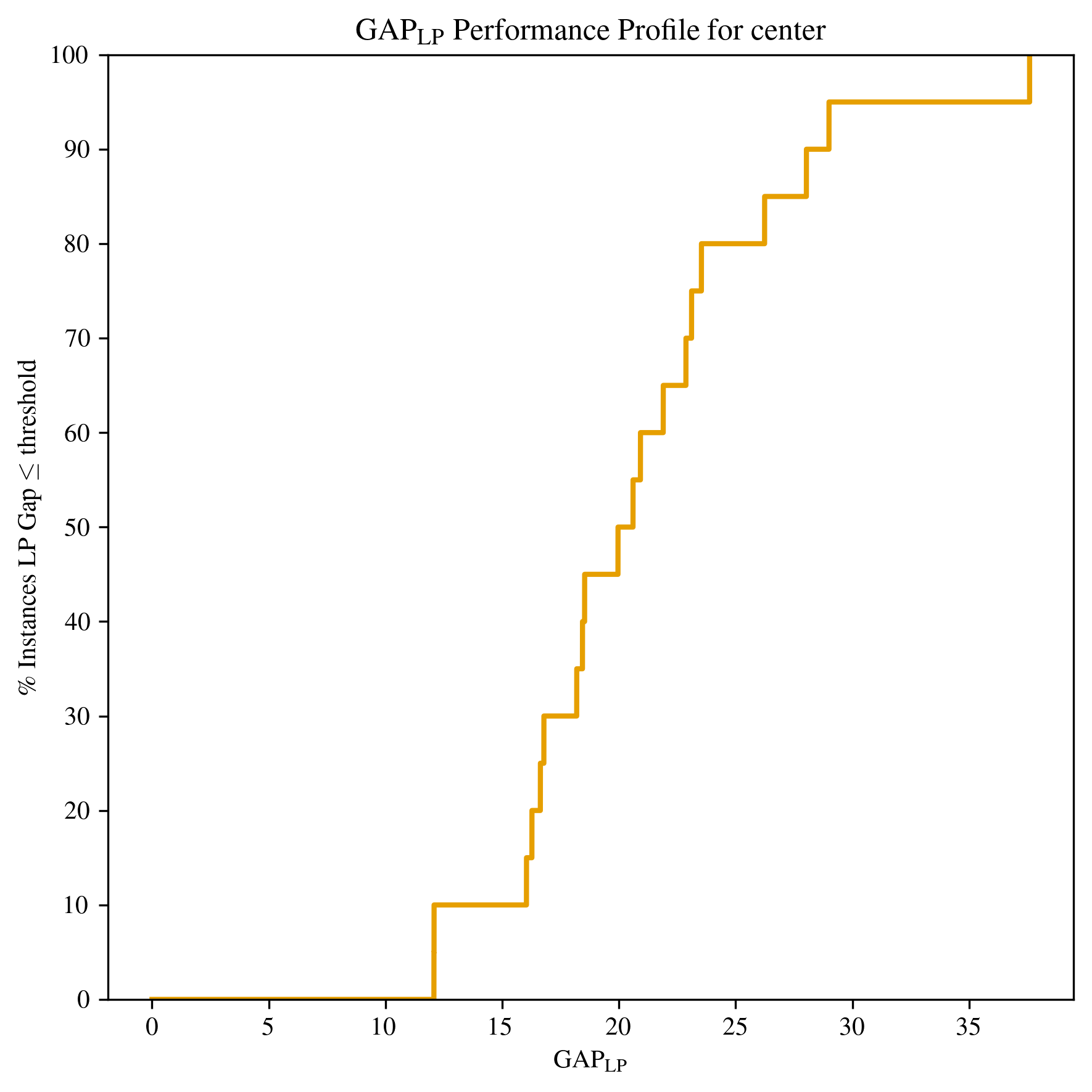}}\\
\adjustbox{width=0.8\textwidth}{\includegraphics{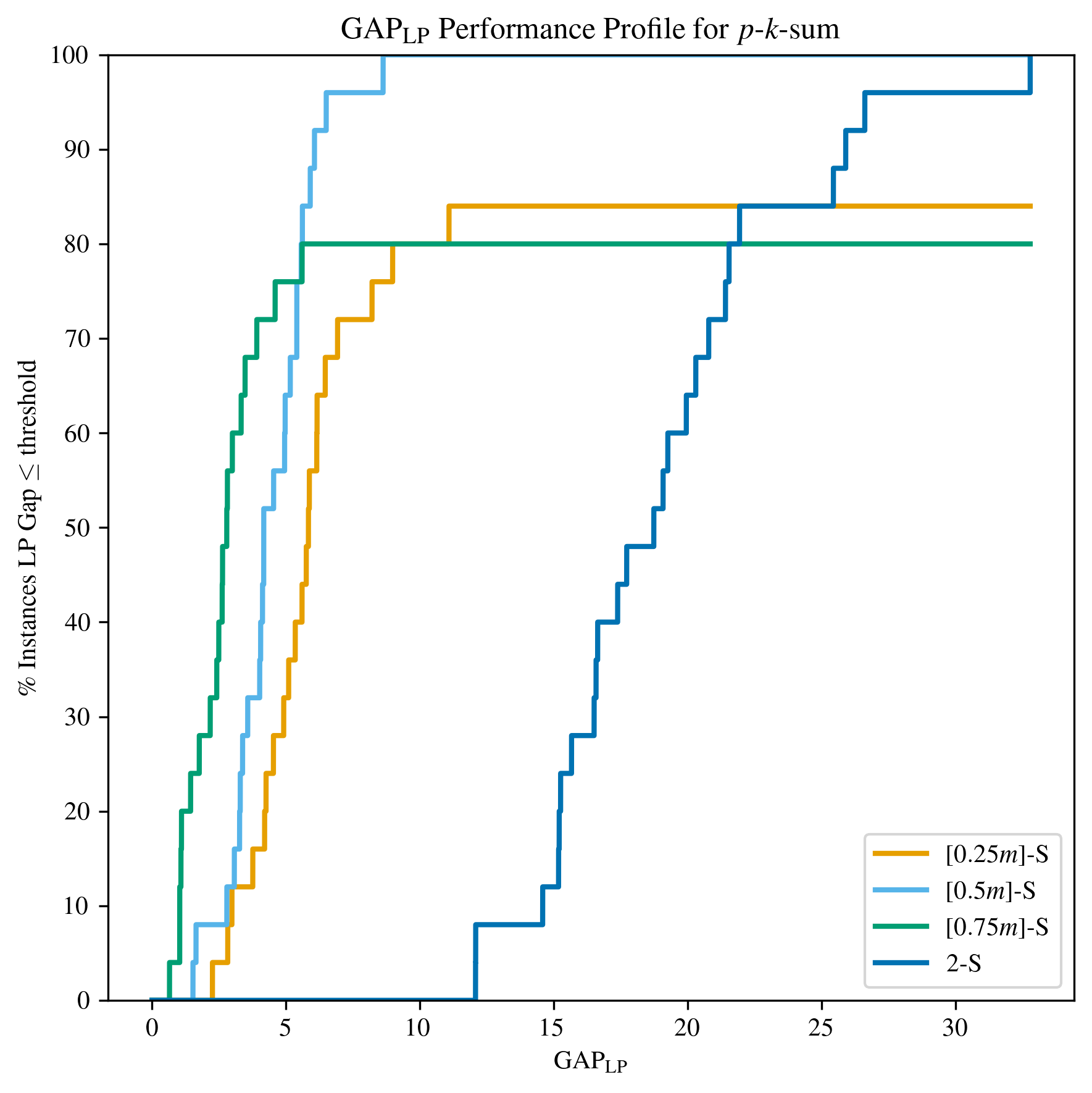}~\includegraphics{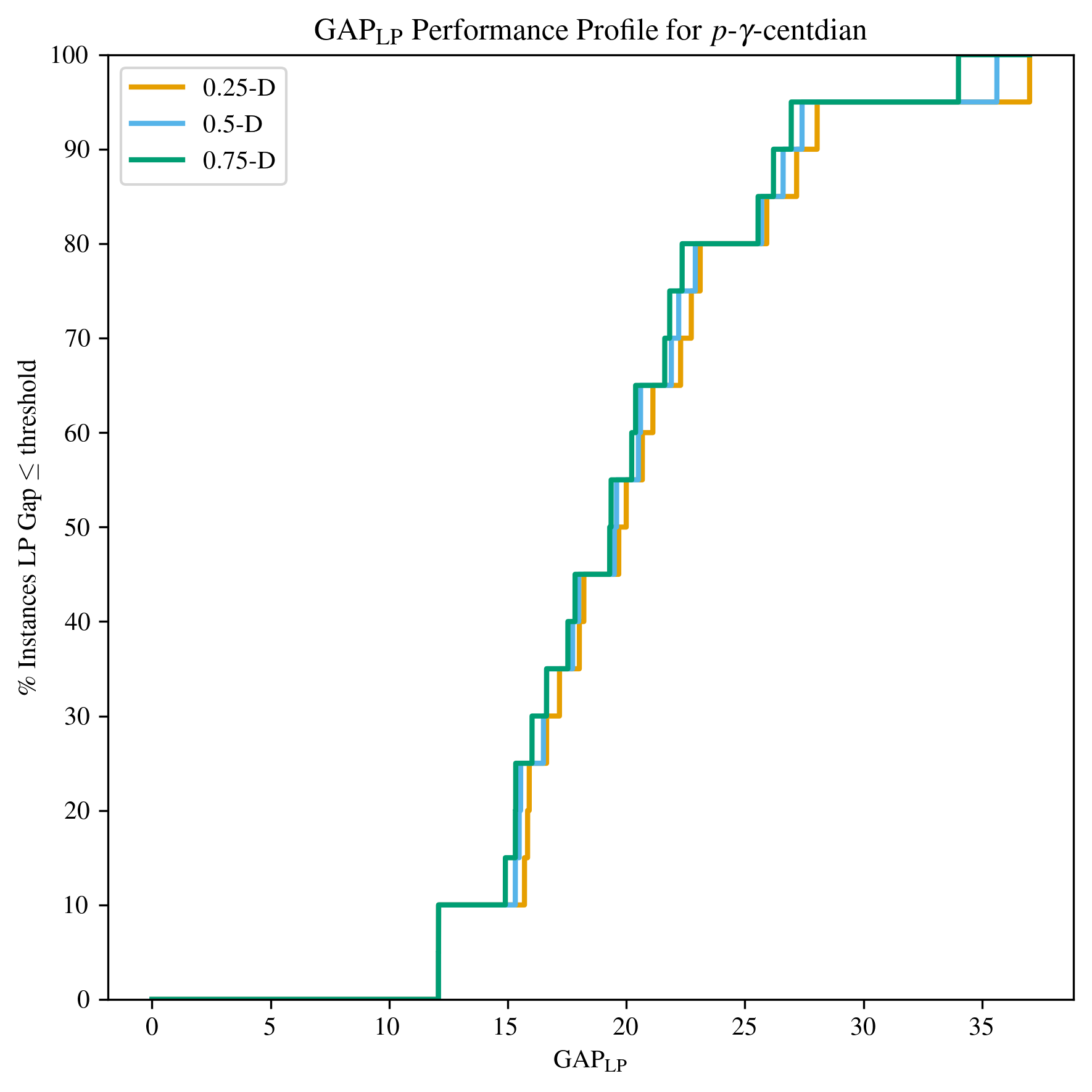}}\\
\caption{Performance profile for LP gaps for all the instances classified by type of DOMP (from top left to bottom right: median, center, sum, and centdian).\label{fig:lpgap_all}}
\end{figure}

One can observe from these figures that the LP relaxation for the $p$-median problem is very close to the MILP objective value, with LP gaps smaller than $0.7\%$ in more than $90\%$ of the instances. These small gaps explain the few nodes that were required to be explored in the branch-and-bound tree in the previous analysis. However, the situation changes dramatically for other versions of the DOMP where sorting is part of the decision process. Specifically, for all the other models, LP gaps larger than $30\%$ are obtained for some instances, especially for the $p$-center, the $2$-sum problem, and all centdian problems. For the sum problems, the larger the number of nonzeros in the $\boldsymbol{\lambda}$-weights, the better the quality of the LP relaxation seems to be. That is, the closer the problem is to the $p$-median problem, the tighter the LP relaxation becomes. 

Note that the range of the values in the $x$-axes of the plots in Figure \ref{fig:lpgap_all} is very different for the types of DOMP analyzed in this study, ranging over $[0\%,0.7\%]$ for the $p$-median problem, but over $[0\%,40\%]$ for the other problems. In order to provide a proper comparison of these LP relaxation gaps, in  Figure~\ref{fig:boxplot_lpgaps} we represent the boxplots of the LP gaps of all instances are represented by type of DOMP, where one can easily check the quality of these lower bounds for the problems in terms of the DOMP.

\begin{figure}[h!]
\centering
\includegraphics[width=0.9\textwidth]{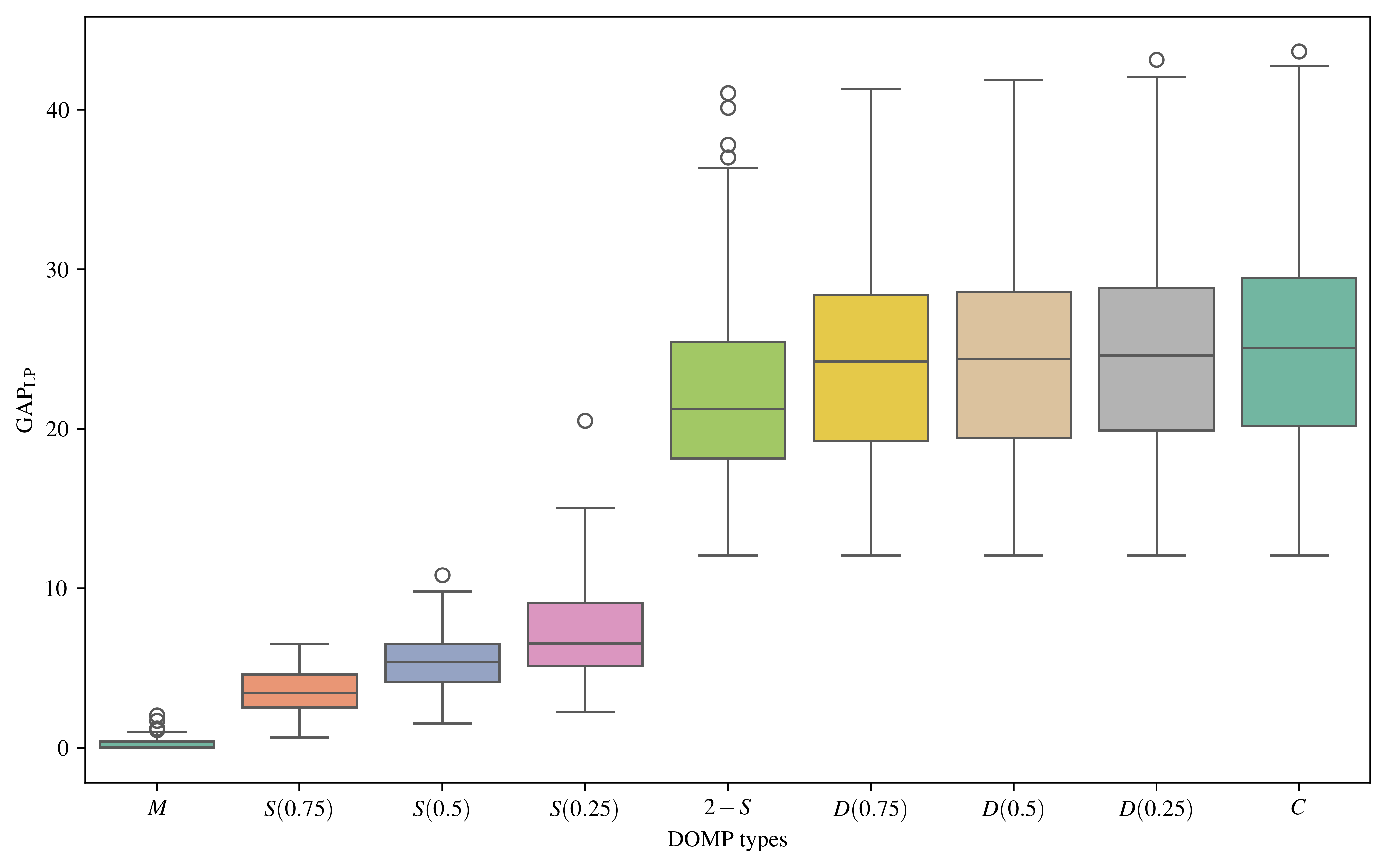}
\caption{Boxplots of LP gaps for all instances by type of DOMP. \label{fig:boxplot_lpgaps}}
\end{figure}

In Table~\ref{tab:lpgaps}, we report the percentage of instances of each DOMP type whose LP relaxation gap is smaller than or equal to $0\%$, $2\%$, $5\%$, and $10\%$. While in the $p$-median problem all instances achieved an LP gap smaller than $5\%$, for many of the other problems (such as the center, centdian, or $2$-sum problems), the gaps never fall below $10\%$. For the sum problems, the LP relaxation of the $75\%$-sum problem is tighter than that of the $50\%$-sum, and the latter, in turn, is better than that obtained for the $25\%$-sum problem.

\begin{table}[h!]
\centering
\adjustbox{width=0.7\textwidth}{
\begin{tabular}{c c c c c}
\toprule
\textbf{DOMP} & \textbf{GAP$_{\rm LP}$(0\%)}& \textbf{GAP$_{\rm LP}$(2\%)}& \textbf{GAP$_{\rm LP}$(5\%)} & \textbf{GAP$_{\rm LP}$(10\%)}\\
\midrule
{\rm M} & $56.99\%$ & $97.85\%$ & $100.00\%$ & $100.00\%$ \\
{\rm C} & $0.00\%$ & $0.00\%$ & $0.00\%$ & $0.00\%$ \\
$\lceil0.25m\rceil$-{\rm S} & $0.00\%$ & $0.00\%$ & $14.94\%$ & $71.26\%$ \\
$\lceil0.5m\rceil$-{\rm S} & $0.00\%$ & $1.94\%$ & $38.83\%$ & $93.20\%$ \\
$\lceil0.75m\rceil$-{\rm S} & $0.00\%$ & $17.86\%$ & $76.19\%$ & $100.00\%$ \\
$2$-{\rm S} & $0.00\%$ & $0.00\%$ & $0.00\%$ & $0.00\%$\\
$0.25$-{\rm D} & $0.00\%$ & $0.00\%$ & $0.00\%$ & $0.00\%$ \\
$0.5$-{\rm D} & $0.00\%$ &$ 0.00\%$ & $0.00\%$ & $0.00\%$ \\
$0.75$-{\rm D} & $0.00\%$ & $0.00\%$ & $0.00\%$ & $0.00\%$ \\
\bottomrule
\end{tabular}}
\caption{Distribution of instances by LP gap smaller or equal than $0\%$, $2\%$, $5\%$, and $10\%$ for each of the DOMP problems solved in our experiments.\label{tab:lpgaps}}
\end{table}
Thus, the theoretical results proved in this paper are empirically supported by this study. In particular, none of the instances of problems different from the $p$-median recovered the MILP solution through its LP relaxation. 

Summarizing this preliminary study, we observe a close relationship between the quality of the LP relaxation and the computational effort required to solve the problem. Additionally, some problems, when solved under the same set of instances, exhibit a consistent behavior inherent to the nature of the $\boldsymbol{\lambda}$-weights defining each DOMP variant. 

\subsection*{LP Relaxations and Clusterability}

As already mentioned, prior studies have analyzed different data generation models (such as the stochastic ball model or the Gaussian mixture model), in which further results can be obtained about the tightening of the LP relaxations in non-ordered optimization. For example, $p$-median problems when the demand points are geometrically distributed around $p$ Euclidean balls whose centers are separated enough. Furthermore, in the clustering literature, it is also well-known that NP-hard clustering algorithms turn into ``easy" problems when instances are clusterizable.

In this part of our computational study, we investigate whether the \emph{clusterability} of the input points has an impact on the quality of the LP relaxation of the different DOMP problems. 

To this end, there are several methodologies designed to detect whether a dataset is clusterizable or not. Among them, we consider Hartigan's dip test~\citep{hartigan1985dip}, whose ability to detect clusterability properties of a dataset has been broadly recognized in literature. It is implemented in \texttt{R} through the library \texttt{clusterabilityR}. We briefly explain the idea behind this test below.

Given an instance for a DOMP, represented through its distance matrix, $\D$, the clusterability test functions that we apply provide a statistical framework to evaluate the presence of cluster structure by analyzing the empirical distribution of the off-diagonal elements in the matrix. The clue is that in datasets containing well-separated groups, the distances tend to be bimodal or multimodal: small intra-cluster distances coexist with large inter-cluster distances. Conversely, a unimodal distance distribution suggests the absence of distinct clusters.

The Hartigan's dip test quantifies the maximum deviation between the \emph{empirical cumulative distribution function} (ECDF) of the distances and the ECDF of the closest unimodal distribution. The resulting \emph{dip statistic} takes values in $[0,1]$, where values near zero point out approximate unimodality and larger values indicate multimodality. 
  The associated $p$-value expresses the probability of observing a 
  dip as extreme under the null hypothesis of unimodality.
  Hence, a small $p$-value suggests that the 
  distance distribution is multimodal, implying that the dataset is 
  likely to contain distinct clusters. A large $p$-value indicates 
  that no statistically significant clustering tendency is detected. To classify a distance matrix as low or highly clusterizable, we use the $5\%$ and $95\%$ quantiles of the two dip measures. Specifically, if an instance has a dip statistic (resp. dip $p$-value) smaller (resp. larger) than the $5\%$ quantile of the corresponding measures in the instance set, it is classified as highly clusterizable. Conversely, if the $95\%$ quantile is used (in the opposite direction), the instance is classified as lowly clusterizable.
  
  For large distance matrices, clusterability was assessed on the one-dimensional classical MDS projection of the data, which retains the main geometric information while providing an interpretable and statistically stable input for unimodality tests.

First, in Figure \ref{fig:pp_dip_all}, we represent the boxplots of the LP relaxation of all solved problems, but classified by level of clusterability (low and high) based on the two metrics that we used, the dip statistic (left) and the dip $p$-value (right).

\begin{figure}[h!]
\centering
\adjustbox{width=0.8\textwidth}{\includegraphics[width=0.5\textwidth]{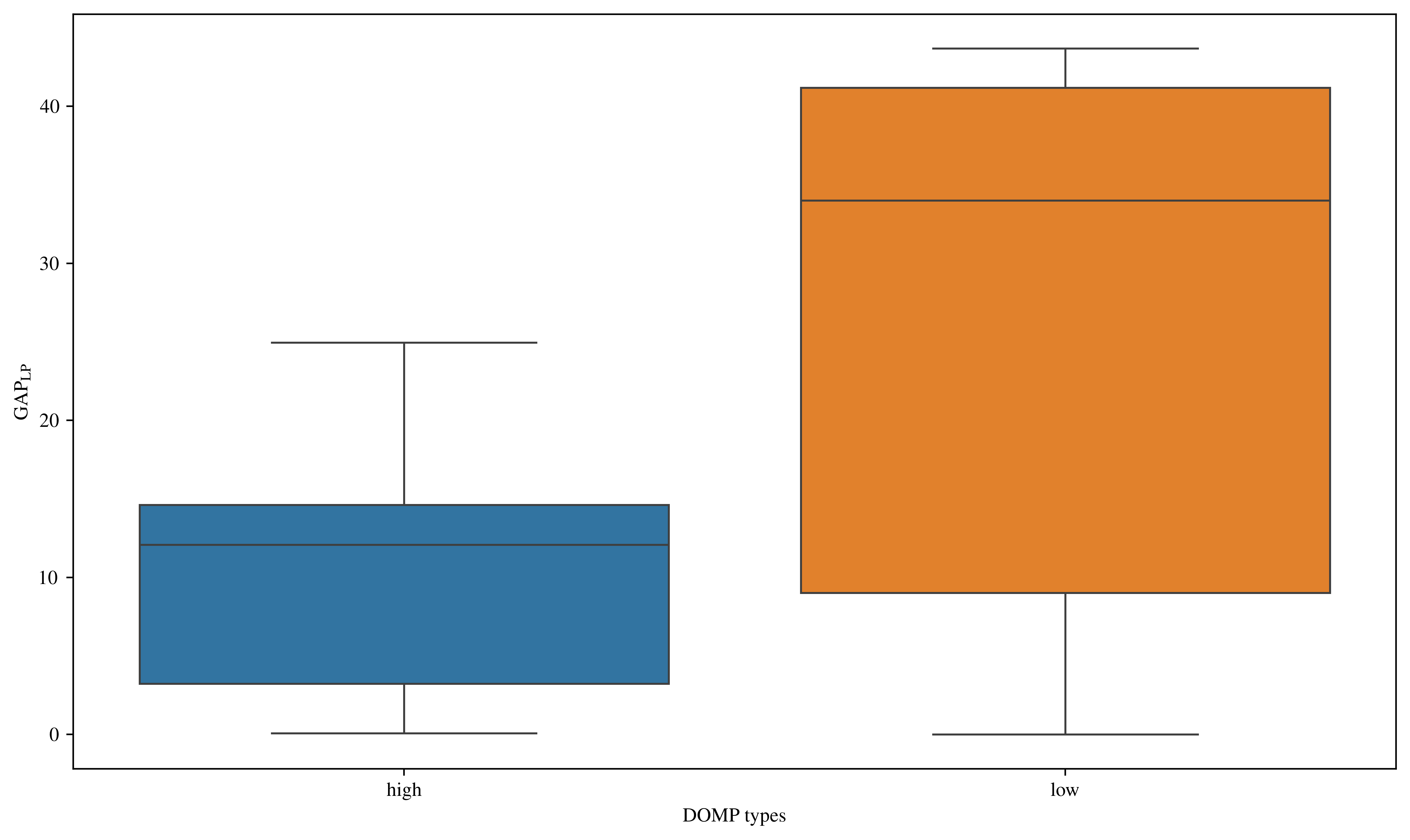}~\includegraphics[width=0.5\textwidth]{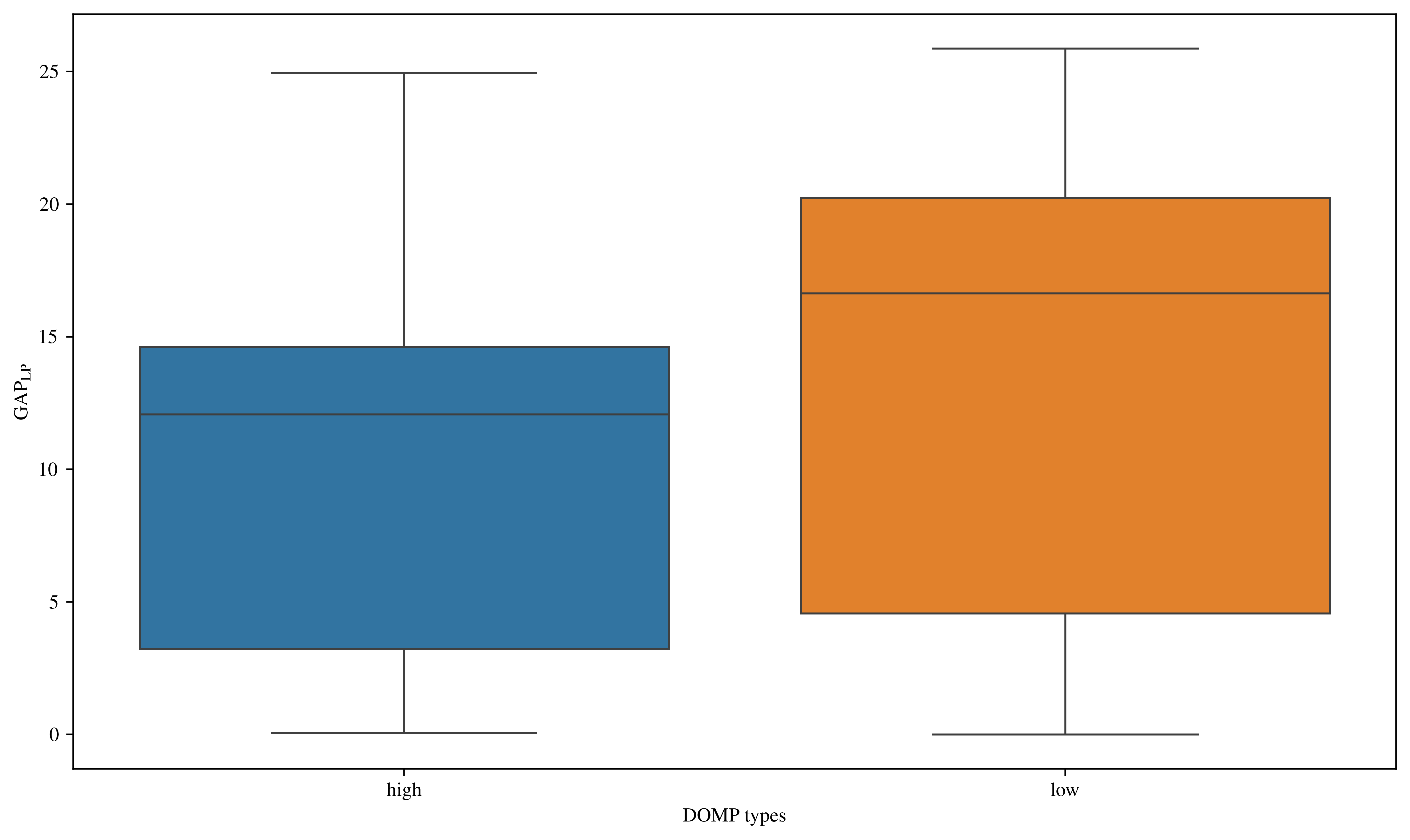}}
\caption{Boxplots of the LP gaps for all the instances classified as either low (orange) or high (blue) clusterability based on the dip statistic (left) and the dip $p$-value (right).\label{fig:pp_dip_all}}
\end{figure}

In general, one can observe, for both measures, that the LP relaxations of the instances that are highly clusterizable according to the dip test perform better than those that are poorly clusterizable, which is consistent with previous studies on the topic. However, when analyzing the results in more detail for each type of DOMP, a differentiated behavior can be observed.

In Figures~\ref{fig:dipstat_problems} and~\ref{fig:dipp_problems}, we present the boxplots for the different problems, according to the clusterability measures, the dip statistic, and the dip $p$-value, respectively. Above each upper whisker, we report the median value of the LP gap. One can observe from the plots that, except for the $p$-median problem, the LP relaxation never recovers the MILP solution, which confirms that sorting is a challenging feature in facility location. 

Moreover, there is a remarkable difference between the quality of the LP gaps when the instances are highly versus poorly clusterizable. In those instances that are suitable for clustering, the LP relaxation is notably closer to the best MILP solution obtained when solving the model. This empirical finding opens a new line of research in facility location that has not been explored before: the impact of the geometric clusterability of the set of demand points on the solvability of $p$-facility location problems. 

It is clear from our results that, in the pursuit of an LP-relaxation-based oracle for ordered median problems, both theoretical and applied insights can be further developed under specific geometric assumptions about the distribution of demand points. Although such studies have been carried out for the $k$-means and $k$-median clustering algorithms, the effect of assuming these geometric properties in ordered problems appears to be even more significant. Indeed, from the boxplots one can see that the medians of the LP gaps differ, in some cases (such as the centdian and center problems), by more than $13\%$.

\begin{figure}[h!]
\centering
\includegraphics[width=0.9\textwidth]{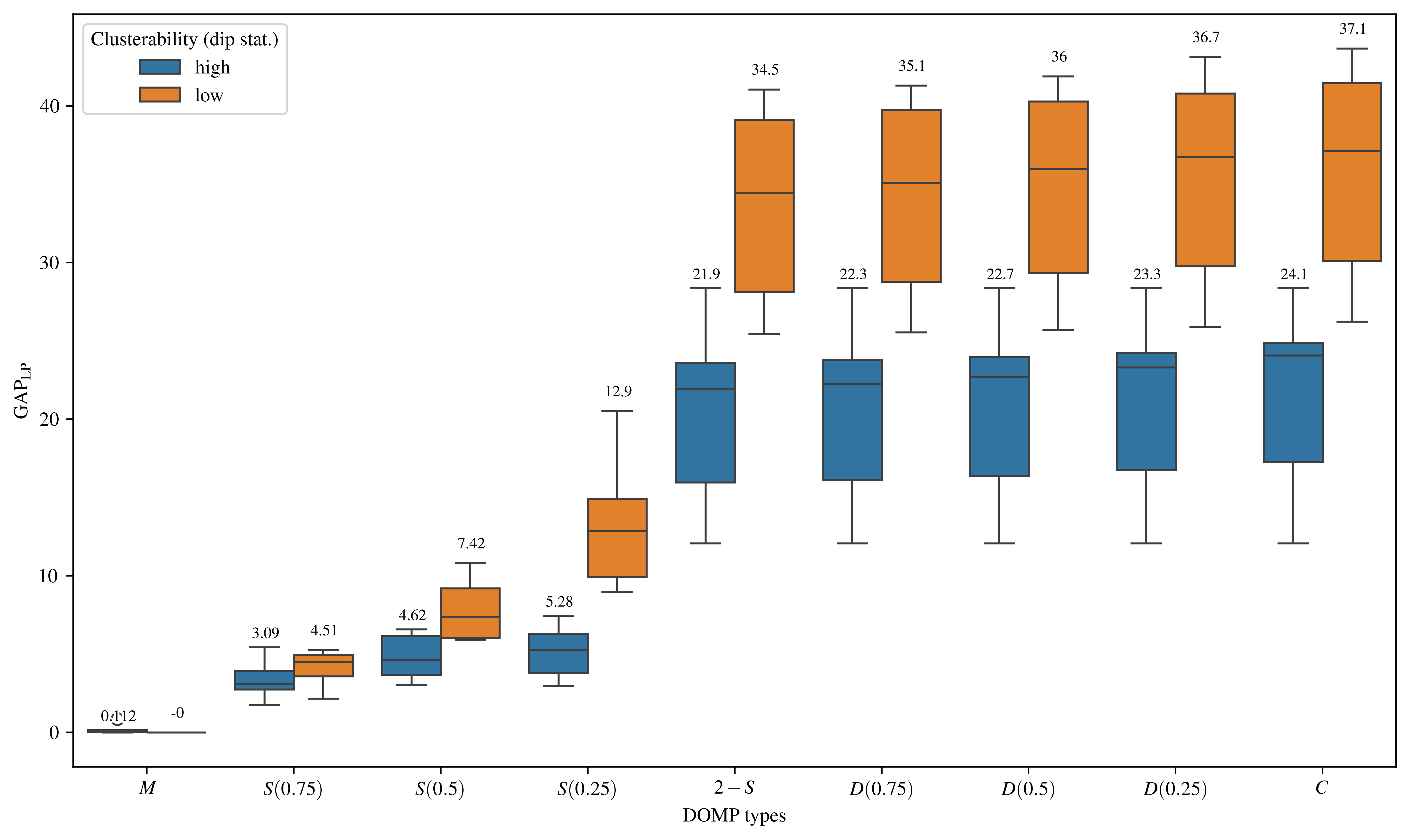}
\caption{Boxplots of LP gaps for all the instances classified as either low (orange) or high (blue) clusterability based on the dip statistic for each type of DOMP. Median values are displayed above the corresponding upper whiskers. \label{fig:dipstat_problems}}
\end{figure}

\begin{figure}[h!]
\centering
\includegraphics[width=0.9\textwidth]{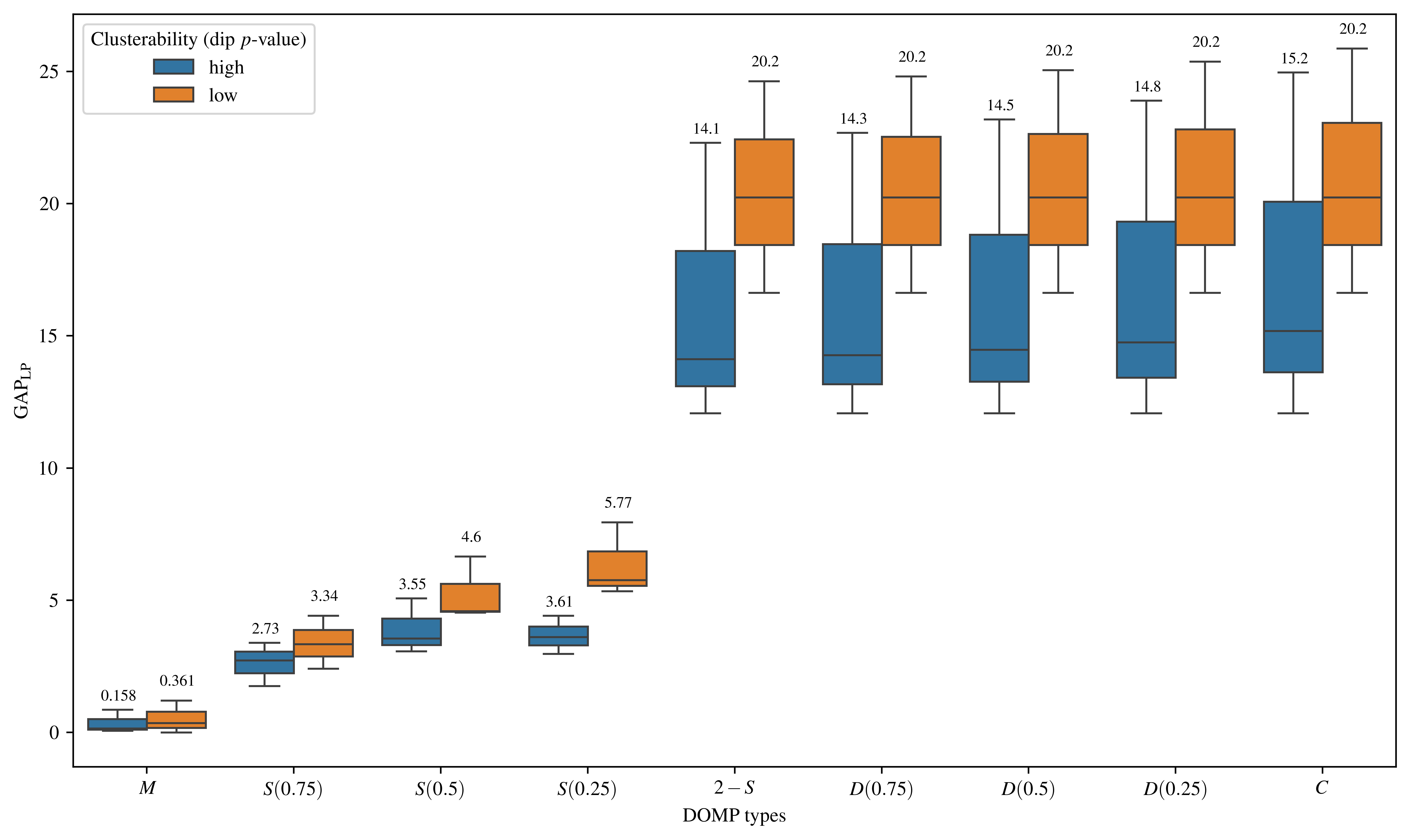}
\caption{Boxplots of LP gaps for all the instances classified as either low (orange) or high (blue) clusterability based on the dip $p$-value for each type of DOMP. Median values are displayed above the corresponding upper whiskers.\label{fig:dipp_problems}}
\end{figure}

\section{Conclusion}\label{sec:conclusion}

In this paper, we have studied, for the first time, the quality of the LP relaxations of ordered location problems, a broad family of models that generalizes the classical $p$-median problem. We derived a novel primal–dual characterization that explains the ability (or failure) of the LP relaxation to recover the integral solution of the problem. We then analyzed several relevant cases within the framework of ordered optimization, namely: the $p$-median problem, the $p$-center problem, the $p$-$\gamma$-centdian problems, and the $p$-$k$-sum problems, concluding that the \emph{favorable} performance of the LP relaxation in the $p$-median problem does not extend to genuinely ordered problems. In particular, we proved that the LP relaxation of the $p$-center problem never recovers integrality, and we provided sufficient non-recovery conditions for the remaining cases.

Finally, we reported an extensive set of computational experiments to analyze the empirical behavior of the LP relaxation on benchmark instances from ORLIB. On the one hand, our computational study confirms that sorting is a challenging feature in mathematical optimization, as clear performance differences arise between the classical $p$-median problem and its ordered counterparts. We also quantified the impact of the LP relaxation on the overall computational effort required to solve each problem. On the other hand, we investigated the influence of geometric properties of the input data, specifically, their degree of clusterability, on the quality of the LP relaxation. We found that this property has a pronounced effect: instances exhibiting high clusterability yield LP relaxations much closer to the optimal integer solutions. This empirical evidence extends and complements existing theoretical results for (non-ordered) optimization-based clustering algorithms.

\noindent{\bf Future Research}

The results obtained in this work not only contribute novel insights into ordered optimization but also open new research directions in location science. First, our findings confirm that current formulations for ordered problems are weakly tightened, and then, further effort is needed to derive valid inequalities or alternative formulations that yield convex relaxations with stronger recovery properties, potentially enabling efficient (polynomial-time) solution methods for problems that are NP-hard in general. Second, we introduced, for the first time, the notion of clusterability in locational analysis and demonstrated, empirically, its implications for the numerical solvability of these problems. A promising avenue for future research is the design of algorithmic strategies that exploit this property, for example, by discarding a subset of points that do not satisfy desirable geometric conditions, solving the simplified problem, and subsequently analyzing the error bounds of such an approximation. Besides, clusterability is a realistic assumption in practical locational settings, where facilities are meant to serve spatially grouped users, since  otherwise, the facility locations obtained would be of limited practical relevance.

A further extension of this research is the analysis of continuous location problems, where facilities can be placed anywhere in the Euclidean space rather than on a predefined discrete set. Although mixed-integer $p$-order cone formulations have been proposed for these models, their convex relaxations are typically weak due to the presence of big-$M$ constraints. Consequently, reformulation techniques and tighter convex relaxations are needed. Advancing this direction would not only strengthen the theoretical foundations of location science but also have direct implications for optimization-based clustering and related fields.

\section*{Acknowledgements}

The authors acknowledge financial support by grants PID2020-114594GB-C21, PID2024-156594NB-C21, and RED2022-134149-T (Thematic Network on Location Science and Related Problems) funded by MICIU/AEI/10.13039/501100011033; FEDER + Junta de Andaluc\'ia project C‐EXP‐139‐UGR23; and the IMAG-Mar\'ia de Maeztu grant CEX2020-001105-M/AEI /10.13039/501100011033.


\end{document}